\documentclass[12pt]{amsart}
\usepackage[left=2.5cm,right=2.5cm, top=2.5cm, bottom=2.5cm]{geometry}
\usepackage{amsmath,amsfonts,amssymb}
\usepackage{color} 
\usepackage{enumerate,graphics, graphicx,tikz}
\usepackage[pdftex,bookmarks=true]{hyperref}
\usepackage{verbatim}
\usepackage{caption}
\usepackage{subcaption}

\usepackage[normalem]{ulem}

\newtheorem{theorem}{Theorem}[section]
\newtheorem{lemma}[theorem]{Lemma}
\newtheorem{conjecture}[theorem]{Conjecture}
\newtheorem{proposition}[theorem]{Proposition}
\newtheorem{prop}[theorem]{Proposition}
\newtheorem{corollary}[theorem]{Corollary}

\theoremstyle{definition}
\newtheorem{defn}[theorem]{Definition}
\newtheorem{definition}[theorem]{Definition}
\newtheorem{eg}[theorem]{Example}
\newtheorem{example}[theorem]{Example}

\newtheorem{question}[theorem]{Question}

\newtheorem{rmk}[theorem]{Remark}
\newtheorem{remark}[theorem]{Remark}

\newcommand{\R}{\mathbb R}

\newcommand{\jac}{\text{Jac}}

\def\been{\begin{enumerate}}
\def\enen{\end{enumerate}}

\title[Identifiability of linear compartmental models]{Identifiability of linear compartmental models: \\
the singular locus}
\author[Gross]{Elizabeth Gross}
\address{University of Hawai`i at M\={a}noa}
\author[Meshkat]{Nicolette Meshkat}
\address{Santa Clara University}
\author[Shiu]{Anne Shiu}
\address{Texas A\&M University}
\date{28 June 2021}

\begin{document}
\maketitle


\begin{abstract}
This work addresses the problem of identifiability, that is, the question of whether parameters can be recovered from data, for linear compartmental models. 
Using standard differential algebra techniques, the question of whether a given model is generically locally identifiable is equivalent to asking whether the Jacobian matrix of a certain coefficient map, arising from input-output equations, is generically full rank. A natural next step is to study the set of parameter values where the Jacobian matrix drops in rank, which we refer to as the \emph{locus of non-identifiable parameter values}, or, for short, the \emph{singular locus}.
 In this work, we give a formula for coefficient maps in terms of acyclic subgraphs of the 
model's underlying 
directed graph and, then, study the case when the singular locus is defined by a single equation, the \emph{singular-locus equation}. We prove that the singular-locus equation can be used to determine when submodels are generically locally identifiable. We also determine the singular-locus equation for two families of linear compartmental models,  cycle and mammillary (star) models with input and output in a single compartment. We also state a conjecture for the corresponding equation for a third family: catenary (path) models.
Finally, we introduce the {\em identifiability degree}, which is the number of parameter values that map to a generic input-output data vector.  This degree was previously computed for mammillary and catenary models, and here we determine this degree for cycle models.   \newline
\end{abstract}

\maketitle



\section{Introduction} \label{sec:intro}
This work focuses on the identifiability problem for linear compartmental models. Linear compartmental models 
are used extensively in
biological applications, such as pharmacokinetics, toxicology, cell biology, physiology, and ecology \cite{Berman1956,Berman1962,distefano-book,godfrey,Mulholland1974}. Indeed, these models 
are now ubiquitous
in pharmacokinetics, with most kinetic parameters for drugs (half-lives, residence times, and so on) based at least in part on linear compartmental model theory \cite{Tozer1981,Wagner1981}.

A mathematical model is {\em identifiable} if its parameters can be recovered from data.  Using standard differential algebra techniques, the question of whether a given linear compartmental model is (generically locally) identifiable is equivalent to asking whether the Jacobian matrix of the {\em coefficient map} (arising from certain {\em input-output equations}) is generically full rank. This work is focused on the set of parameter values where the Jacobian matrix of the coefficient map drops in rank.  This set is an algebraic subvariety of the parameter space, which we call the \emph{locus of non-identifiable parameter values}, or, for short, the \emph{singular locus}.  While the singular locus can be informative in respect to establishing identifiable parameters values and submodels, to the best of our knowledge, this variety associated to linear compartmental models has not yet been studied in depth. 

If the generic rank of the Jacobian matrix of the coefficient map is $k$, then the singular locus is defined by all $(k+1) \times (k+1)$ minors of the Jacobian. However, we find that, in the case of several popular biological models, the singular locus is defined by a single equation.  We call this  
equation the \emph{equation of the singular locus}.  The equation of the singular locus allows us to fully understand the set of identifiable parameters, enabling us to give biologically relevant information regarding the identifiability of parameters. For example, for the cycle and mamillary models studied in Section \ref{sec:singular}, the singular-locus equation implies that every parameter value with nonzero and unique entries is identifiable.  Furthermore, in more general settings, we can use the singular-locus equation to find generically locally identifiable submodels.  Indeed, using the singular-locus equation to give information on which edges of a model can be deleted while preserving identifiability (Theorem~\ref{thm:delete}) is our first main result.

Identifiable submodels have been studied by Vajda and others \cite{vajda1984, vajdaetal}, and the exploration of when a submodel of an identifiable model is identifiable is important both on the theoretical side and on the applied side.  On the theoretical side, this problem is one step in addressing the large, overarching question in regards to linear compartmental models: {\em 
Can we determine whether a model is identifiable by simply inspecting its underlying directed graph?} On the applied side, the operation of removing an edge can correspond to a biological intervention, such as a genetic knockout or a drug that inhibits a specific activity.

After looking at the role the singular locus has in determining identifiable submodels, we move toward a combinatorial inspection of the coefficient maps of linear compartmental models. Understanding the coefficient map is a first step in understanding the singular locus for a given model. Recently, Meshkat, Sullivant, and Eisenberg gave a general formula for the input-output equations
of linear compartmental models~\cite{MeshkatSullivantEisenberg}.  
Our second main result takes their formula for the coefficient map~\cite{MeshkatSullivantEisenberg} and recasts it in terms of combinatorial properties of the model, namely in terms of forests (acyclic subgraphs) in the associated directed graph  (Theorem~\ref{thm:coeff-i-o-general}).  This formula allows us to determine the singular-locus equation for the models explored in the later sections.

Finally, our remaining results pertain to three well-known families of linear compartmental models, which are depicted in Figures~\ref{fig:cat} and~\ref{fig:cycle}: catenary (path graph) models, mammillary (star graph) models, and cycle models~\cite{godfrey}.  

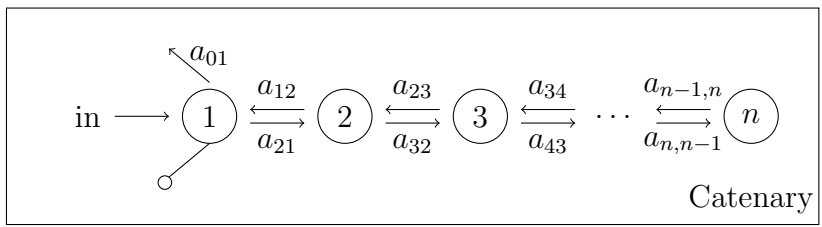
\begin{figure}[ht]
\begin{center}
	\begin{tikzpicture}[scale=1.8]
 	\draw (0,0) circle (0.2);	
 	\draw (1,0) circle (0.2);	
 	\draw (2,0) circle (0.2);	
 	\draw (4,0) circle (0.2);	
    	\node[] at (0, 0) {1};
    	\node[] at (1, 0) {$2$};
    	\node[] at (2, 0) {$3$};
    	\node[] at (3, 0) {$\dots$};
    	\node[] at (4, 0) {$n$};
	 \draw[<-] (0.3, 0.05) -- (0.7, 0.05);	
	 \draw[->] (0.3, -0.05) -- (0.7, -0.05);	
	 \draw[<-] (1.3, 0.05) -- (1.7, 0.05);	
	 \draw[->] (1.3, -0.05) -- (1.7, -0.05);	
	 \draw[<-] (2.3, 0.05) -- (2.7, 0.05);	
	 \draw[->] (2.3, -0.05) -- (2.7, -0.05);	
	 \draw[<-] (3.3, 0.05) -- (3.7, 0.05);	
	 \draw[->] (3.3, -0.05) -- (3.7, -0.05);	
   	 \node[] at (0.5, 0.2) {$a_{12}$};
   	 \node[] at (1.5, 0.2) {$a_{23}$};
   	 \node[] at (2.5, 0.2) {$a_{34}$};
   	 \node[] at (3.5, 0.2) {$a_{n-1,n}$};
   	 \node[] at (0.5, -0.2) {$a_{21}$};
   	 \node[] at (1.5, -0.2) {$a_{32}$};
   	 \node[] at (2.5, -0.2) {$a_{43}$};
   	 \node[] at (3.5, -0.2) {$a_{n,n-1}$};
 	\draw (-0.33,-.49) circle (0.05);	
	 \draw[-] (0, -.2 ) -- (-0.3, -.45);	
	 \draw[->] (-0.7, 0) -- (-0.3, 0);	
   	 \node[] at (-.9, 0) {in};
	 \draw[->] (0, .25) -- (-0.3, .5);	
   	 \node[] at (0, 0.45) {$a_{01}$};
    	\node[above] at (4, -0.8) {Catenary};
\draw (-1.5,-.8) rectangle (4.5, .8);
	\end{tikzpicture}
\end{center}
\caption{The {\bf catenary} (path) model with $n$ compartments, in which compartment 1 has an input, output, and leak.  See Section~\ref{sec:model}.
}
\label{fig:cat}
\end{figure}

\begin{figure}[ht]
\begin{center}
	\begin{tikzpicture}[scale=1.8]
 	\draw (-1,0) circle (0.2);	
 	\draw (-0.5,0.866) circle (0.2);	
 	\draw (0.5,-0.866) circle (0.2);	
 	\draw (-0.5,-0.866) circle (0.2);	
	 \draw[->] (-0.85, -0.25) -- (-0.62, -0.64 );
 	 \draw[->]  (-0.2,-0.866) -- (0.2,-0.866) ;
	 \draw[->]  (0.62, -0.64) -- (0.9,-0.17); 
	 \draw[loosely dotted,thick] (0.9,0.17) -- (0.5,0.866);
	 \draw[->]  (0.2,0.866) -- (-0.2,0.866);
	 \draw[->] (-0.62, 0.64 ) -- (-0.85, 0.25)  ;
   	 \node[] at (0.0, -1.03) {$a_{32}$};
   	 \node[] at (0.0, 1.03) {$a_{n,n-1}$};
   	 \node[] at (-0.55, -0.42) {$a_{21}$};
   	 \node[] at (-0.55, 0.4) {$a_{1n}$};
   	 \node[] at (0.55, -0.42) {$a_{43}$};
    	\node[] at (-1, 0) {1};
    	\node[] at  (-0.5,-0.866) {$2$};
    	\node[] at  (+0.5,-0.866) {$3$};
    	\node[] at (-0.5,0.866) {$n$};
 	\draw (-1.33,-.49) circle (0.05);	
	 \draw[-] (-1, -.2 ) -- (-1.3, -.45);	
	 \draw[->] (-1.7, 0) -- (-1.3, 0);	
   	 \node[] at (-1.9, 0) {in};
	 \draw[->] (-1.1, .25) -- (-1.4, .5);	
   	 \node[] at (-1.1, 0.45) {$a_{01}$};
    	\node[above] at (-1.8, -1.5) {Cycle};
\draw (-2.2,-1.5) rectangle (1.4, 1.4);
 	\draw (3,0) circle (0.2);
 	\draw (5,-1) circle (0.2);
 	\draw (5,-0.3) circle (0.2);
 	\draw (5,1) circle (0.2);
    	\node[] at (3, 0) {1};
    	\node[] at (5, -1) {2};
    	\node[] at (5, -0.3) {3};
    	\node[] at (5, 0.3) {$\vdots$};
    	\node[] at (5, 1) {$n$};
	 \draw[->] (3.25, -.2) -- (4.65, -1);
	 \draw[<-] (3.25, -.1) -- (4.65, -0.9);
	 \draw[->] (3.25, -0.05) -- (4.65, -0.3);
	 \draw[<-] (3.25, 0.05) -- (4.65, -0.2);
	 \draw[->] (3.25, 0.1) -- (4.6, 0.9);
	 \draw[<-] (3.25, 0.2) -- (4.6, 1);
   	 \node[] at (4.1, -0.9) {$a_{21}$};
   	 \node[] at (4.57, -0.7) {$a_{12}$};
   	 \node[] at (4.57, -0.05) {$a_{13}$};
   	 \node[] at (4.1, -0.35) {$a_{31}$};
   	 \node[] at (4, 0.9) {$a_{1,n}$};
   	 \node[] at (4.5, 0.55) {$a_{n,1}$};
 	\draw (2.67,-.49) circle (0.05);	
	 \draw[-] (3, -.2 ) -- (2.7, -.45);	
	 \draw[->] (2.3, 0) -- (2.7, 0);	
   	 \node[] at (2.1, 0) {in};
	 \draw[->] (2.9, .25) -- (2.6, .5);	
   	 \node[] at (2.9, 0.45) {$a_{01}$};
\draw (1.7,-1.5) rectangle (5.5, 1.4);
    	\node[above] at (2.4, -1.5) {Mammillary};
	\end{tikzpicture}
\end{center}
\caption{Two models with $n$ compartments, where compartment 1 has an input, output, and leak.
{ Left:} The {\bf cycle}.  { Right:} The {\bf mammillary} (star). 
}
\label{fig:cycle}
\end{figure}
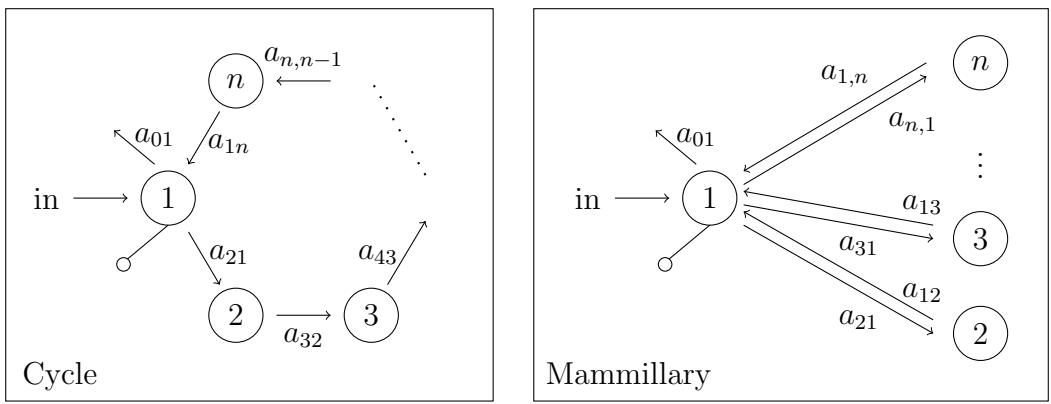

For these three families of models,
which are (generically locally) identifiable~\cite{cobelli-constrained-1979,MeshkatSullivant,MeshkatSullivantEisenberg},
we obtain results and a conjecture, summarized in Table~\ref{tab:summary}, on:
\begin{enumerate}
\item the equation of the singular locus, and
\item the {\em identifiability degree}: this degree is $m$ if exactly $m$ sets of parameter values map to a generic input-output data vector.
\end{enumerate}

\begin{table}[h]
\centering
\begin{tabular}{| l | c | c |}
\hline
{\bf Model }                                                & {\bf Equation of singular locus} & {\bf Identifiability} \\
 & & {\bf degree} \\
\hline
Catenary (path)    &       	
{\rm Conjecture:~}
$a_{12}^{n-1} 
	(a_{21} a_{23})^{n-2} 
	\dots
	(a_{n-1,n-2} a_{n-1,n})
$                      &         $1$
      \\ \hline
Cycle           &     	
$	a_{32}a_{43} \dots a_{n,n-1} a_{1,n}
	\prod_{2 \leq i < j \leq n}
		\left( a_{i+1,i} - a_{j+1,j}
			\right) $
                        &             $(n-1)!$    \\      \hline
Mammillary (star)  &    $	a_{12} a_{13} \dots a_{1,n} 
	\prod_{2 \leq i < j \leq n} \left(a_{1i} - a_{1j} \right)^2$
                         &         $(n-1)!$               \\ \hline
\end{tabular}
\caption{Summary of theorems, conjectures, and prior results on the 
linear compartmental models depicted in Figures~\ref{fig:cat} and~\ref{fig:cycle}.  
See Theorems~\ref{thm:mam},~\ref{thm:cycle}, and~\ref{thm:ident-degree-cycle}; 
Proposition~\ref{prop:prior-ident-degree}; 
and Conjecture~\ref{conj:cat}.
Note that the $n=2$ versions of these three models coincide, so their equations and degrees agree.
Also, note that the singular locus for all these models is a union of hyperplanes (e.g., $a_{12}=a_{13}$), 
including some coordinate hyperplanes (e.g., $a_{12}=0$).
}
\label{tab:summary}
\end{table}

The outline of our work is as follows.
In Section~\ref{sec:background}, we introduce linear compartmental models and define the singular locus. 
In Section~\ref{sec:submodels}, we prove our result on how the singular locus gives information on identifiable submodels.
In Section~\ref{sec:jac}, we give a new combinatorial formula for the coefficients of the input-output equations for linear compartmental models with input and output in a single compartment.
We use this formula to prove, in Sections~\ref{sec:singular}
and~\ref{sec:iddegree}, the results on the singular-locus equations and identifiability degrees mentioned above for the models in Figures~\ref{fig:cat} and~\ref{fig:cycle}.
We conclude with a discussion in Section~\ref{sec:discussion}.

\section{Background} \label{sec:background}
In this section, we recall linear compartmental models, their input-output equations, and the concept of identifiability. We also introduce the main focus our work: the locus of non-identifiable parameter values (the singular locus) and the equation that defines it.

\subsection{Linear compartmental models} \label{sec:model}

A {\em linear compartmental model} consists of a directed graph 
$G = (V,E)$ 
together with three sets $In, Out,$ $Leak \subseteq V$.
Each vertex $i \in V$ corresponds to a
compartment in the model and each edge $j \rightarrow i$ corresponds to  
a direct flow of material from the $j$-th compartment to the
$i$-th compartment. The sets 
$In, Out, Leak \subseteq V$ are the
sets of input compartments, output compartments, and leak compartments, 
respectively. We always assume that $Out \neq \emptyset$, as models without outputs are not identifiable.

Following the literature, we will indicate output compartments by this symbol: \begin{tikzpicture}[scale=0.7]
 	\draw (4.66,-.49) circle (0.05);	
	 \draw[-] (5, -.15) -- (4.7, -.45);	
\end{tikzpicture} .  Input compartments are labeled by ``in'', and leaks are indicated by outgoing edges.  
For instance, each of the linear compartmental models depicted in Figures~\ref{fig:cat} and~\ref{fig:cycle} have $In, Out, Leak =\{1\}$.

To each edge $j \rightarrow i$ of $G$, we associate
a 
parameter $a_{ij}$, the rate of flow
from compartment $j$ to compartment $i$.  
To each leak node $i \in Leak$, we associate a 
parameter $a_{0i}$, the rate of flow from compartment $i$ leaving the system.  Let $n=|V|$.
The {\em compartmental matrix} of a linear compartmental model $(G,In, Out, Leak)$ is the $n \times n$ matrix 
 $A$ 
 with entries given by:
\[
  A_{ij} 
  ~:=~ \left\{ 
  \begin{array}{l l l}
    -a_{0i}-\sum_{k: i \rightarrow k \in E}{a_{ki}} & \quad \text{if $i=j$ and } i \in Leak\\
        -\sum_{k: i \rightarrow k \in E}{a_{ki}} & \quad \text{if $i=j$ and } i \notin Leak\\
    a_{ij} & \quad \text{if $j\rightarrow{i}$ is an edge of $G$}\\
    0 & \quad \text{otherwise.}\\
  \end{array} \right.
\]

A linear compartmental model $(G, In, Out, Leak)$ defines a system of linear ODEs (with inputs $u_i(t)$) and 
outputs $y_{i}(t)$ as follows:
\begin{align} \label{eq:main}
x'(t) ~&=~ Ax(t)+u(t) \\ 
y_i(t) ~&=~ x_i(t) \quad \quad \mbox{ for } i \in Out~, \notag
\end{align}
 where $u_{i}(t) \equiv 0$ for $i \notin In$.
 
We now define the concepts of $\textit{strongly connected}$ and $\textit{inductively strongly connected}$. 

\begin{defn} \label{def:strongly-connected} ~
\begin{enumerate}
	\item A directed graph $G$ is \textit{strongly connected} if there exists a directed path from each vertex to every other vertex.  
	A directed graph $G$ is \textit{inductively strongly connected} with respect to vertex $1$ if
	there is an ordering of the vertices $1,\ldots,|V|$ that starts at vertex $1$ 
	such that each of the induced subgraphs $G_{\{1, \ldots, i\}}$ is strongly connected for $i = 1, \ldots, |V|$. 
	\item A linear compartmental model $(G, In, Out, Leak)$ is \textit{strongly connected} (respectively,  \textit{inductively strongly connected}) if $G$ is strongly connected (respectively, inductively strongly connected).
\end{enumerate}
\end{defn}

The two most common classes of 
compartmental models are \textit{mammillary} (star) and \textit{catenary} (path) model structures (see Figures~\ref{fig:cat} and~\ref{fig:cycle}).  Mammillary models consist of a central compartment surrounded by and connected with peripheral (noncentral) compartments, none of which are connected to each other \cite{distefano-book}.  Catenary models have all compartments arranged in a chain, with each connected (in series) only to its nearest neighbors \cite{distefano-book}.   In a typical pharmacokinetic application, the central compartment of a mammillary model consists of blood plasma and 
the peripheral compartments correspond to
highly perfused tissues in which a drug distributes rapidly.  For catenary models, the drug distributes more slowly. For examples of how mammillary and catenary models are used in practice, see \cite{distefano-book, godfrey, vicini}.

Another common class of 
compartmental models is formed by \textit{cycle} models (see Figure~\ref{fig:cycle}).  A cycle model consists of a single directed cycle. 
Cycle models are only strongly connected, while mammillary and catenary models are inductively strongly connected. For examples of cycle models, see \cite{audolydangio,egritoth,vajda1982}.

\subsection{Input-output equations} \label{sec:i-o}
 The {\em input-output equations} of a linear compartmental model are equations that hold along every solution of the ODEs~\eqref{eq:main}, and which involve only the parameters $a_{ij}$, input variables $u_i$, output variables $y_i$, and their derivatives.  
A general form of these equations was given by 
 Meshkat, Sullivant, and Eisenberg~\cite[Corollary 1]{MeshkatSullivantEisenberg} for the case of strongly connected models and a generalization was given in \cite[Proposition 2.3]{linear-i-o}. 
 The version of this result we state here is 
 for the case of one input and one output:
\begin{proposition}[Meshkat, Sullivant, and Eisenberg] \label{prop:MSE}
	Consider a linear compartmental model that 
	has an input in compartment $j$ and an output in compartment $i$ (and no other inputs or outputs).
	Let $A$ denote the compartmental matrix, 
	let $\partial$ be the differential operator $d/dt$, and let $(\partial I-A)_{ij}$ denote the submatrix of $(\partial I-A)$ obtained by removing row $i$ and column $j$.  Then an 
	input-output equation 
	is the following:
	\begin{align} \label{eq:i-o-general}
	\det (\partial I - A) y_i ~=~  \det \left( (\partial I-A)_{ij} \right) u_j~. 
	\end{align}
\end{proposition}

\begin{example} \label{ex:prop-MSE}
Consider the following catenary model (the $n=3$ case from Figure~\ref{fig:cat}):
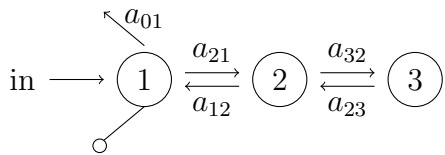
\begin{center}
	\begin{tikzpicture}[scale=1.8]
 	\draw (0,0) circle (0.2);	
 	\draw (1,0) circle (0.2);	
 	\draw (2,0) circle (0.2);	
    	\node[] at (0, 0) {1};
    	\node[] at (1, 0) {$2$};
    	\node[] at (2, 0) {$3$};
	 \draw[->] (0.3, 0.05) -- (0.7, 0.05);	
	 \draw[<-] (0.3, -0.05) -- (0.7, -0.05);	
	 \draw[->] (1.3, 0.05) -- (1.7, 0.05);	
	 \draw[<-] (1.3, -0.05) -- (1.7, -0.05);	
   	 \node[] at (0.5, 0.2) {$a_{21}$};
   	 \node[] at (1.5, 0.2) {$a_{32}$};
   	 \node[] at (0.5, -0.2) {$a_{12}$};
   	 \node[] at (1.5, -0.2) {$a_{23}$};
 	\draw (-0.33,-.49) circle (0.05);	
	 \draw[-] (0, -.2 ) -- (-0.3, -.45);	
	 \draw[->] (-0.7, 0) -- (-0.3, 0);	
   	 \node[] at (-.9, 0) {in};
	 \draw[->] (0, .25) -- (-0.3, .5);	
   	 \node[] at (0, 0.45) {$a_{01}$};
	\end{tikzpicture}
\end{center}
By Proposition~\ref{prop:MSE}, an input-output equation is:
\[
\det 
	\begin{pmatrix}
	d/dt + a_{01} + a_{21} & - a_{12} & 0 \\
	-a_{21} & d/dt +a_{12}+ a_{32} & -a_{23}\\
	0 & -a_{32} & d/dt + a_{23} \\
	\end{pmatrix} y_1
~=~
\det 
	\begin{pmatrix}
	 d/dt +a_{12}+ a_{32} & -a_{23}\\
	 -a_{32} & d/dt + a_{23} \\
	\end{pmatrix} u_1	~,
\]
which, when expanded, becomes:
{\footnotesize
\begin{align} \notag
	& y_1^{(3)}
	+
	\left( a_{01} + a_{12} + a_{21} + a_{23}+ a_{32}
		\right) y_1^{(2)}
	+
	\left( a_{01}a_{12} +  a_{01}a_{23} + a_{01}a_{32} + a_{12}a_{23} + a_{21}a_{23} + a_{21}a_{32} 
		\right) y_1'
	+
		\left( a_{01} a_{12} a_{23}
		\right) y_1 \\
& \quad \quad \quad ~=~
	 u_1^{(2)}
	+
	\left( a_{12}+a_{23}+a_{32}
		\right) u_1'
	+
	\left( a_{12} a_{23}
		\right) u_1~.
		\label{eq:i-o-equation-for-example}
\end{align}
}
Observe, from the left-hand side of equation~\eqref{eq:i-o-equation-for-example} 
 that the coefficient of $y_1^{(i)}$ corresponds to the set of forests (acyclic subgraphs) of the model that have $(3-i)$ edges and at most 1 outgoing edge per compartment.
As for the right-hand side, the coefficient of $u_1^{(i)}$ corresponds to similar $(n-i-1)$-edge forests in the following model:
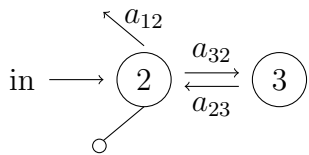
\begin{center}
	\begin{tikzpicture}[scale=1.8]
 	\draw (0,0) circle (0.2);	
 	\draw (1,0) circle (0.2);	
    	\node[] at (0, 0) {2};
    	\node[] at (1, 0) {$3$};
	 \draw[->] (0.3, 0.05) -- (0.7, 0.05);	
	 \draw[<-] (0.3, -0.05) -- (0.7, -0.05);	
   	 \node[] at (0.5, 0.2) {$a_{32}$};
   	 \node[] at (0.5, -0.2) {$a_{23}$};
 	\draw (-0.33,-.49) circle (0.05);	
	 \draw[-] (0, -.2 ) -- (-0.3, -.45);	
	 \draw[->] (-0.7, 0) -- (-0.3, 0);	
   	 \node[] at (-.9, 0) {in};
	 \draw[->] (0, .25) -- (-0.3, .5);	
   	 \node[] at (0, 0.45) {$a_{12}$};
	\end{tikzpicture}
\end{center}
This combinatorial interpretation of the coefficients of the input-output equation generalizes, as we will see in Theorem~\ref{thm:coeff-i-o-general}.
\end{example}

\subsection{Identifiability}
A linear compartmental model is 
\textit{generically structurally identifiable} if from a generic choice of the inputs
and initial conditions, the parameters of the model can be recovered 
from exact measurements of both the inputs and the outputs.  Recent work from Ovchinnikov, Pogudin, and Thompson showed that checking identifiability from the input-output equation in~\eqref{eq:i-o-general} is valid for strongly connected models with one input and one output \cite[Corollary 2]{Ovchinnikov-Pogudin-Thompson}.
We now define this concept precisely. 

\begin{defn}\label{defn:identify}
Let $\mathcal{M}=(G, In, Out, Leak)$ be a strongly connected linear compartmental model with 
one input and one output. 
The  \textit{coefficient map}
is the function 
$c:  \R^{|E| + |Leak|}  \rightarrow \R^{k}$
that is the vector of all coefficient functions of
the input-output equation in~\eqref{eq:i-o-general}
(here $k$ is the total number of coefficients).
Then
$\mathcal{M}$ is:
\begin{enumerate}
	\item \textit{globally identifiable} if $c$ is one-to-one, and is \textit{generically globally identifiable} if $c$ is one-to-one outside a set of measure zero.
	\item {\textit{locally identifiable} if
		around every point in $\R^{|E| + |Leak|}$ there is an open neighborhood $U$ such that 
		 $c : U  \rightarrow \R^k$  is one-to-one, and is \textit{generically locally identifiable} if, 
		 outside a set of measure zero, every point in $\R^{|E| + |Leak|}$ has such an open 
		 neighborhood $U$.}
\item{\textit{unidentifiable} if $c$ is infinite-to-one.}
\end{enumerate}
\end{defn}

Since the coefficients in $c$ are all polynomial functions of the
parameters, the model $\mathcal{M}=(G, In, Out, Leak)$ is generically locally
identifiable if and only if the image of $c$ has dimension equal to
the number of parameters, i.e.,~$|E| + |Leak|$.  The dimension of the image of
a map is equal to the rank of the Jacobian matrix at a generic point.  
Thus we have the following result, which is \cite[Proposition 2]{MeshkatSullivantEisenberg}:

\begin{prop} [Meshkat, Sullivant, and Eisenberg] \label{prop:jacobian}
A linear compartmental model $(G, In,$ $Out, Leak)$ 
is generically locally
identifiable if and only if 
the rank of
the Jacobian matrix of its coefficient map $c$, 
when evaluated at a generic point, 
is equal to $|E| + |Leak|$. 
\end{prop}

\begin{example} For the model in Example \ref{ex:prop-MSE}, 
the input-output equation 
was shown in~\eqref{eq:i-o-equation-for-example}.  The coefficient map $c: \mathbb{R}^5 \to \mathbb{R}^5$ is therefore given by:
\begin{multline*}
(a_{01},a_{12},a_{21},a_{23},a_{32}) \mapsto  \\
(a_{01} + a_{12} + a_{21} + a_{23}+ a_{32}
	,~
	 a_{01}a_{12} +  a_{01}a_{23} + a_{01}a_{32} + a_{12}a_{23} + a_{21}a_{23} + a_{21}a_{32},~
	  a_{01} a_{12} a_{23}, ~a_{12}+a_{23}+a_{32},
	  ~
	   a_{12} a_{23})~.
\end{multline*}
The $5 \times 5$ Jacobian matrix of $c$ has determinant equal to $-a_{12}^2 a_{21} a_{23}$ and so, by Proposition~\ref{prop:jacobian}, the model is generically locally identifiable.  
This identifiability result is well known for general catenary models (see Proposition \ref{prop:prior-ident-degree} below), and a general formula for the determinant of the Jacobian matrix (for catenary models) is conjectured in a later section (Conjecture~\ref{conj:cat}).
\end{example}

\begin{rmk} 
An alternative to Proposition~\ref{prop:jacobian} is to test identifiability by using a Gr\"obner basis to solve the system of equations $c(p)=c(p^*)$, where $p^*$ is an arbitrary point in the parameter space, for $p$.  The model is globally identifiable if there is a unique solution $p=p^*$, locally identifiable if there are a finite number of solutions, and unidentifiable if there are an infinite number of solutions.  In practice, Gr\"obner basis computations are more computationally expensive than Jacobian calculations (as in Proposition~\ref{prop:jacobian}).
\end{rmk}


We now examine 
when the Jacobian is {\em generically} full rank, but certain parameter choices lead to rank-deficiency.  We call parameter values that lead to this rank-deficiency {\em non-identifiable}.  Note that the parameters of these models are {\em generically} identifiable, and in the identifiability literature are called ``identifiable'' \cite{distefano-book}, but for our purposes, we are examining the non-generic case and thus denote the {\em values} of these parameters ``non-identifiable''.    

\begin{defn} \label{def:sing-locus}
Let $\mathcal{M}=(G, In, Out, Leak)$ be a strongly connected linear compartmental model, with one input and one output, that is generically locally identifiable.
Let $c$ denote the coefficient map.
The {\em locus of non-identifiable parameter values}, or, for short, the \textit{singular locus} 
is the subset of the parameter space $\R^{|E| + |Leak|}$ where the Jacobian matrix of $c$ has rank strictly less than $|E| + |Leak|$. 
\end{defn}

Thus, the singular locus is the defined by the set of all $(|E| + |Leak|)  \times (|E| + |Leak|)$ minors of $\jac(c)$.  
We will focus on the cases when only a single such minor, which we give a name to below, defines the singular locus:
\begin{defn} \label{def:eq-sing-locus}
Let $\mathcal{M}=(G, In, Out, Leak)$ be a linear compartmental model,
with coefficient map $c:  \R^{|E| + |Leak|}  \rightarrow \R^{k}$.
Suppose $\mathcal{M}$ is generically locally identifiable (so, $|E| + |Leak| \leq k$).
\begin{enumerate}
	\item If $|E| + |Leak| = k$ (the number of parameters equals the number of coefficients), then $\det (\jac(c))$ is the {\em equation of the singular locus}.
	\item Assume $|E| + |Leak| < k$. Suppose
	  there is a choice of $|E| + |Leak|$ coefficients from~$c$, with 
	  $r:\R^{|E| + |Leak|}  \rightarrow \R^{|E| + |Leak|}$ the resulting restricted coefficient map, 
	  such that $\det (\jac(r))=0$ if and only if the Jacobian of $c$ has rank strictly less than $|E| + |Leak|$.
	Then $\det (\jac(r))$ is 
	the {\em equation of the singular locus}.
\end{enumerate}
\end{defn}

\begin{rmk} The equation of the singular locus,
when  $|E| + |Leak| = k$, is defined only up to sign, as we do not specify 
the order of the coefficients in $c$.
When $|E| + |Leak| < k$, 
there need not be a single  $(|E| + |Leak|)  \times (|E| + |Leak|)$ minor that defines the singular locus,
and thus a singular-locus equation as defined above might not exist.  
We, however, have not encountered such a model, although we suspect one exists.
Accordingly, we ask, {\em is there always a choice of coefficients or, equivalently, rows of $\jac(c)$, such that this square submatrix is rank-deficient if and only if the original matrix $\jac(c)$ is?}  And, when such a choice exists, {\em is this choice of coefficients unique?} 
%
\end{rmk}

\begin{rmk}
In applications, we typically are only interested in the factors of the singular-locus equation: we only care whether, e.g., $a_{12}$ divides the equation (i.e., whether $a_{12}$ is non-identifiable) and not which higher powers $a_{12}^m$, for positive integers $m$, also divide it.
\end{rmk}

One aim of our work is to investigate the equation of the singular locus for mammillary, catenary, and cycle models with a single input, output, and leak in the first compartment.  As a start, all of these families of models are (at least) generically locally identifiable:

\begin{prop}  \label{prop:gli}
The $n$-compartment catenary, cycle, and mammillary models in Figures~\ref{fig:cat} and~\ref{fig:cycle} (with input, output, and leak in compartment 1 only) are generically locally identifiable.
\end{prop}

\begin{proof}  
Catenary and mammillary models are inductively strongly connected with $2n-2$ edges.  Thus,  catenary and mammillary models with a single input and output in the first compartment and leaks from every compartment 
have coefficient maps with images of 
maximal dimension~\cite[Theorem 5.13]{MeshkatSullivant}.  Removing all the leaks except one from the first compartment, we can apply \cite[Theorem 1]{MeshkatSullivantEisenberg} and obtain generic local identifiability.  

Similarly, \cite[Proposition 5.4]{MeshkatSullivant} implies that the image of the coefficient map for cycle models with leaks from every compartment has maximal dimension. Thus removing all leaks except one results in a generically locally identifiable model, again by applying \cite[Theorem 1]{MeshkatSullivantEisenberg}.   
\end{proof}

In fact, catenary models are generically {\em globally} identifiable (Proposition~\ref{prop:prior-ident-degree}).   We also will investigate the identifiability degrees of the other two models, and in particular, prove that the cycle model has 
identifiability degree $(n-1)!$ (Theorem~\ref{thm:ident-degree-cycle}).

\section{The singular locus and identifiable submodels} \label{sec:submodels}
One reason a model's singular locus is of interest is because it gives us information regarding the identifiability of particular parameter values.  Indeed, for generically locally identifiable models, 
the singular locus contains the set of parameter values that cannot be recovered, even locally. 
A second reason for studying the singular locus, which is the main focus of this section, is that the singular-locus equation gives information about which submodels are identifiable.

\begin{theorem}[Identifiable submodels] \label{thm:delete}
Let $\mathcal{M}=(G, In, Out, Leak)$ 
be a linear compartmental model that 
is strongly connected and generically locally identifiable, 
has an input and output in compartment 1 (and no other inputs or outputs), 
and has singular-locus equation $f$.
Let $\widetilde {\mathcal{M}}$ be the model obtained from $\mathcal{M}$ by deleting a set of edges $\mathcal{I}$ of $G$. 
If $\widetilde {\mathcal{M}}$ is strongly connected, and 
$f$ is {\em not} in the ideal $\langle a_{ji} \mid (i,j) \in \mathcal{I} \rangle$ (or, equivalently, after evaluating $f$ at $a_{ji}=0$ for all $(i,j) \in \mathcal{I}$, the resulting polynomial is nonzero),
 then 
$\widetilde { \mathcal{M}}$ is generically locally identifiable. 
\end{theorem}

\begin{proof} 
Let $\mathcal{M}=(G, In, Out, Leak)$,
the submodel $\widetilde {\mathcal{M}}$, 
the polynomial
	$f \in \mathbb{Q}[a_{ij} \mid (i,j) \in E(G), ~{\rm or}~ i=0 ~{\rm and}~j \in Leak]$,
and the subset
 $\mathcal{I} \subseteq E(G)$ be as in the statement of the theorem. 
Thus, the following polynomial $\widetilde f$, obtained by evaluating $f$ at $a_{ji}=0$ for all deleted edges $(i,j)$ in  $\mathcal{I}$, is {\em not} the zero polynomial:
\begin{align*}
	\widetilde f ~:=~ f|_{a_{ji}=0 \text{ for } (i,j) \in \mathcal{I}} ~\in ~ \mathbb{Q}[a_{ji} \mid (i,j) \in E(G) \setminus \mathcal{I} , ~{\rm or}~ j=0 ~{\rm with}~i \in Leak ]~.
\end{align*}
In addition,
$f=\det \jac(r)$, where $r: \mathbb{R}^m \to \mathbb{R}^m$ is a choice of 
$m:=\lvert E(G) \rvert+ \lvert Leak  \rvert$ 
coefficients from $\mathcal{M}$'s input-output equation~\eqref{eq:i-o-general} in Proposition~\ref{prop:MSE}.  
  
Let $\widetilde m = m - \lvert \mathcal{I} \rvert$.  
%
Let $J:= \jac(r)|_{a_{ji}=0 \text{ for } (i,j) \in \mathcal{I}} $ denote the 
matrix obtained from $\jac (r)$ by setting $a_{ji}=0$ for all $(i,j) \in \mathcal{I}$.  
The determinant of $J$ is the nonzero polynomial $\widetilde f$, 
so $J$ is full rank when evaluated at any parameter vector $(a_{ji})$ 
outside the measure-zero set $V(\widetilde f) \subseteq \mathbb{R}^{\widetilde m}$.  
(Here, $V(f)$ denotes the real vanishing set of $f$.)
Thus, the $m \times \widetilde m$ matrix $B$ obtained from $J$ by deleting the set of columns corresponding to $\mathcal{I}$,
 is also full rank ($\mathop{rank}(B)=\widetilde m$) outside of $V(\widetilde f) \subseteq \mathbb{R}^{\widetilde m}$.

Choose $\widetilde m$ rows of $B$ that are linearly independent outside some measure-zero set in $\mathbb{R}^{\widetilde m}$.  (Such a choice exists, because, otherwise, $B$ would be rank-deficient on all of $\mathbb{R}^{\widetilde m}$ and thus so would the generically full-rank matrix $J$, which is a contradiction.)  These rows form an ${\widetilde m} \times {\widetilde m}$ matrix that we call $\widetilde J$.

Let $\widetilde r: \mathbb{R}^{\widetilde m} \to \mathbb{R}^{\widetilde m}$ be obtained from 
$r$ by restricting to the coordinates $r_i$ corresponding to the above choice of rows of $B$, and also setting $a_{ji}=0$ for all $(i,j) \in \mathcal{I}$.
By construction and by Proposition~\ref{prop:MSE} 
(here we use that $In=Out=\{1\}$), $\widetilde{r}$ is 
a choice of $\widetilde m$ coefficients from the input-output equations of $\widetilde{\mathcal{M}}$, and, by construction, the Jacobian matrix of $\widetilde{r}$ is $\widetilde J$ (whose rows we chose to be generically full rank).  Hence, $\widetilde{\mathcal{M}}$ is generically locally identifiable.
\end{proof}

\begin{eg} \label{ex:delete}
Consider the following (strongly connected) linear compartmental model $\mathcal{M}$: 
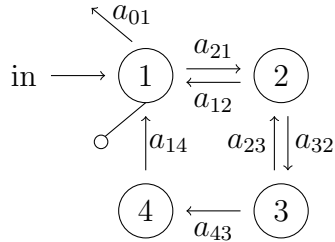
\begin{center}
	\begin{tikzpicture}[scale=1.8]
 	\draw (-1,0) circle (0.2);	
 	\draw (0,0) circle (0.2);	
 	\draw (0,-1) circle (0.2);	
 	\draw (-1,-1) circle (0.2);	
	 \draw[->] (-0.7, 0.05) -- (-0.3, 0.05);	
	 \draw[<-] (-0.7, -0.05) -- (-0.3, -0.05);	
	 \draw[<-] (-0.7, -1) -- (-0.3, -1);	
	 \draw[->] (-1, -0.7) -- (-1,-0.3);	
	 \draw[<-] (0.05, -0.7) -- (0.05,-0.3);	
	 \draw[->] (-0.05, -0.7) -- (-0.05,-0.3);	
   	 \node[] at (-0.5, 0.2) {$a_{21}$};
   	 \node[] at (-0.5, -0.2) {$a_{12}$};
   	 \node[] at (-0.5, -1.15) {$a_{43}$};
   	 \node[] at (0.25, -0.5) {$a_{32}$};
   	 \node[] at (-0.25, -0.5) {$a_{23}$};
   	 \node[] at (-0.83, -0.5) {$a_{14}$};
    	\node[] at (-1, 0) {1};
    	\node[] at (0, 0) {2};
    	\node[] at (0,-1) {3};
    	\node[] at (-1, -1) {4};
 	\draw (-1.33,-.49) circle (0.05);	
	 \draw[-] (-1, -.2 ) -- (-1.3, -.45);	
	 \draw[->] (-1.7, 0) -- (-1.3, 0);	
   	 \node[] at (-1.9, 0) {in};
	 \draw[->] (-1.1, .25) -- (-1.4, .5);	
   	 \node[] at (-1.1, 0.45) {$a_{01}$};
	 \end{tikzpicture}
\end{center}

\noindent
This model is generically locally identifiable, and the equation of the singular locus is:
{\footnotesize
\[ a_{12} a_{14} a_{21}^2 a_{32} (a_{12} a_{14} - a_{14}^2 - a_{12} a_{23} + a_{14} a_{23} + a_{14} a_{32} - 
   a_{12} a_{43} + a_{14} a_{43} - a_{32} a_{43}) (a_{12} a_{23} + a_{12} a_{43} + a_{32} a_{43})~.
\] }This equation is {\em not} divisible by $a_{23}$, and the model  $\widetilde{\mathcal{M}}$ obtained by removing that edge (labeled by $a_{23}$) is strongly connected.  So, by Theorem~\ref{thm:delete}, 
 $\widetilde{\mathcal{M}}$
is generically locally identifiable.
\end{eg}

The converse of Theorem~\ref{thm:delete} does not hold, as we see in the following example.

\begin{eg}[Counterexample to converse of Theorem~\ref{thm:delete}] \label{ex:converse}
Consider again the model $\mathcal{M}$ from Example~\ref{ex:delete}.  
The submodel obtained by deleting the edges labeled by $ a_{12}$ and $a_{23}$ is generically locally identifiable (by Theorem~\ref{thm:cycle} below: the submodel is the 4-compartment cycle model).
Nevertheless, the singular-locus equation of $\mathcal{M}$ is divisible by $a_{12}$ and thus the equation is in the ideal $\langle a_{12}, a_{23} \rangle$. 
\end{eg}

Example~\ref{ex:converse},
our counterexample to the converse of Theorem~\ref{thm:delete}, involved deleting two edges ($ \lvert \mathcal{I} \rvert =2$).  We do not know of a counterexample that deletes only one edge, and we end this section with the following question.

\begin{question} \label{q:converse}
In the setting of Theorem~\ref{thm:delete}, if a parameter $a_{ij}$ divides $f$, does it follow that the model $\mathcal{M}'$ obtained by deleting the edge labeled by $a_{ij}$ is unidentifiable  (assuming that $\mathcal{M}'$ is strongly connected)?
\end{question}

\begin{remark} \label{rmk:CJSSS}
Question~\ref{q:converse} was pursued recently by Chan {\em et al.}, who conjectured an affirmative answer~\cite{CJSSS}.
\end{remark}

\section{The coefficient map and its Jacobian matrix} \label{sec:jac}

Recall that for linear compartmental models with one input and one output, an input-output equation was given in equation~\eqref{eq:i-o-general} 
 (in Proposition~\ref{prop:MSE}).  In this section, we give a new combinatorial formula for the coefficients of this equation (Theorem~\ref{thm:coeff-i-o-general}).

\subsection{Preliminaries} \label{sec:gphs}
To state Theorem~\ref{thm:coeff-i-o-general}, we must define some graphs associated to a model. In what follows, we use ``graph'' to mean ``directed graph''. 
\begin{definition} \label{def:gphs}
Consider a linear compartmental model $\mathcal{M} = (G, In, Out, Leak)$ with $n$~compartments.
	\begin{enumerate}
	\item The {\em leak-augmented graph} of $\mathcal{M}$, denoted by $\widetilde G$, is obtained from $G$ by adding a new node, labeled by 0, and adding edges $j \to 0$ labeled by $a_{0j}$, for every leak $j \in Leak$.
	\item The graph $ \widetilde G_i$, for some $i=1,\dots, n$, is obtained from $\widetilde G$ by completing these steps:
		\begin{itemize}
		\item Delete compartment $i$, by taking the induced subgraph of $\widetilde G$ with vertices \\ $\{0,1,\dots, n\} \setminus \{i\}$, and then:
		\item For each edge $j \to i$ (with label $a_{ij}$) in $\widetilde G$, 
		if $j \in Leak$ (i.e, $j \to 0$ with label $a_{0j}$ is an edge in $\widetilde G$), then label the leak $j \to 0$ in $\widetilde G_i$ by $(a_{0j} + a_{ij})$; if, on the other hand, $j \notin Leak$, then add to $ \widetilde G_i$ the edge $j \to 0$ with label $a_{ij}$.
		\end{itemize}
	\end{enumerate}
 \end{definition}

\begin{example} \label{ex:gphs}
Figure~\ref{fig:gphs} displays a model $\mathcal{M}$, its leak-augmented graph $\widetilde G$, and the graphs $\widetilde G_1$ and $\widetilde G_2$.  
The compartmental matrix of $\mathcal{M}$ is:
\[ A ~=~
	\begin{pmatrix}
	-a_{01}-a_{21} & a_{12}\\
	a_{21} & -a_{12} \\
	\end{pmatrix}~.
\]
The compartmental matrix that corresponds to $\widetilde G_1$ is obtained from $A$ by removing row 1 and column 1.  Similarly, for $\widetilde G_2$, the corresponding compartmental matrix comes from deleting row 2 and column 2 from $A$.
This observation generalizes (see Lemma~\ref{lem:remove-1}).

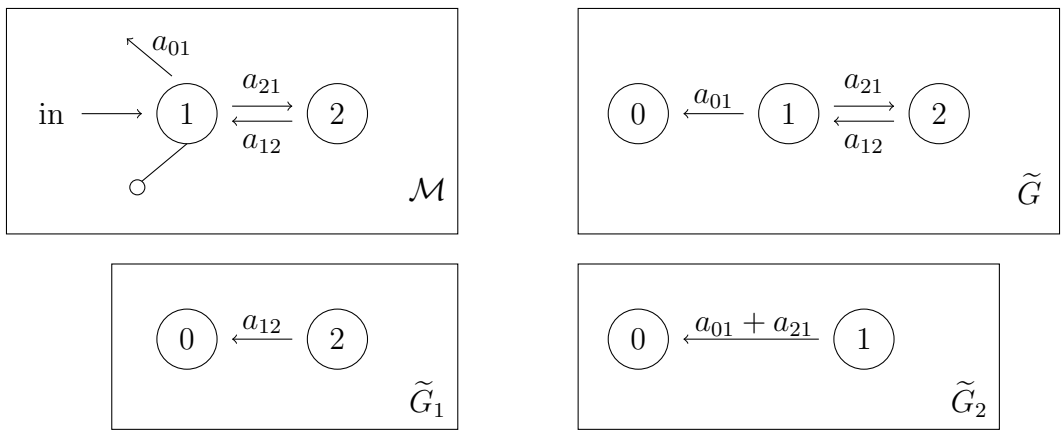
\begin{figure}[ht] 
\begin{center}
	\begin{tikzpicture}[scale=2]
 	\draw (-1,0) circle (0.2);	
 	\draw (0,0) circle (0.2);	
	 \draw[->] (-0.7, 0.05) -- (-0.3, 0.05);	
	 \draw[<-] (-0.7, -0.05) -- (-0.3, -0.05);	
   	 \node[] at (-0.5, 0.2) {$a_{21}$};
   	 \node[] at (-0.5, -0.2) {$a_{12}$};
    	\node[] at (-1, 0) {1};
    	\node[] at (0, 0) {2};
 	\draw (-1.33,-.49) circle (0.05);	
	 \draw[-] (-1, -.2 ) -- (-1.3, -.45);	
	 \draw[->] (-1.7, 0) -- (-1.3, 0);	
   	 \node[] at (-1.9, 0) {in};
	 \draw[->] (-1.1, .25) -- (-1.4, .5);	
   	 \node[] at (-1.1, 0.45) {$a_{01}$};
\draw (-2.2,-.8) rectangle (0.8, .7);
   	 \node[] at (0.6, -0.5) {$\mathcal{M}$};
 	\draw (2,0) circle (0.2);	
 	\draw (3,0) circle (0.2);	
 	\draw (4,0) circle (0.2);	
	 \draw[->] (3.3, 0.05) -- (3.7, 0.05);	
	 \draw[<-] (3.3, -0.05) -- (3.7, -0.05);	
	 \draw[<-] (2.3, 0) -- (2.7, 0);	
   	 \node[] at (3.5, 0.2) {$a_{21}$};
   	 \node[] at (3.5, -0.2) {$a_{12}$};
   	 \node[] at (2.5, 0.1) {$a_{01}$};
	\node[] at (2,0) {0};
    	\node[] at (3, 0) {1};
    	\node[] at (4, 0) {2};
\draw (1.6,-.8) rectangle (4.8, .7);
   	 \node[] at (4.6, -0.5) {$\widetilde G$};
 	\draw (-1,-1.5) circle (0.2);	
 	\draw (0,-1.5) circle (0.2);	
	 \draw[<-] (-0.7, -1.5) -- (-0.3, -1.5);	
   	 \node[] at (-0.5, -1.4) {$a_{12}$};
    	\node[] at (-1, -1.5) {0};
    	\node[] at (0, -1.5) {2};
\draw (-1.5,-2.1) rectangle (0.8, -1);
   	 \node[] at (0.6, -1.9) {$\widetilde G_1$};
	\draw (2,-1.5) circle (0.2);	
 	\draw (3.5,-1.5) circle (0.2);	
	 \draw[<-] (2.3, -1.5) -- (3.2, -1.5);	
   	 \node[] at (2.77, -1.4) { $a_{01}+a_{21}$};
	\node[] at (2,-1.5) {0};
    	\node[] at (3.5, -1.5) {1};
\draw (1.6,-2.1) rectangle (4.4, -1);
   	 \node[] at (4.2, -1.9) {$\widetilde G_2$};
	 \end{tikzpicture}
\end{center}
\caption{A model $\mathcal{M}$, its leak-augmented graph $\widetilde G$, and the graphs $\widetilde G_1$ and $\widetilde G_2$. } \label{fig:gphs}
\end{figure}
\end{example}

\begin{lemma} \label{lem:remove-1}
Consider a linear compartmental model $\mathcal{M}$ with compartmental matrix $A$ and $n$ compartments.  Let 
 $ \widetilde G_i$, for some $i=1,\dots, n$, be as in Definition~\ref{def:gphs}.  Then for any model $\mathcal{M}'$ whose leak-augmented graph is $\widetilde G_i$, the compartmental matrix of $\mathcal{M}'$ is  the matrix obtained from $A$ by removing row $i$ and column $i$.
\end{lemma}

\begin{proof}
Let $\mathcal{M}$, $A$, and $\mathcal{M}'$ be as in the statement of the lemma.  
Let $A_{ii}$ denote the matrix obtained from $A$ by removing row $i$ and column $i$. 
We must show that the compartmental matrix of $\mathcal{M}'$ equals $A_{ii}$.
The graph $ \widetilde G_i$ is obtained by taking the induced subgraph of $\widetilde{G}$ formed by all vertices except $i$
 -- which ensures that the off-diagonal entries of the compartmental matrix of $\mathcal{M}'$ equal those of $A_{ii}$ --
and then 
replacing edges directed toward $i$ with leak edges (and combining them as necessary with existing leak edges) 
-- which ensures that the diagonal entries of the compartmental matrix also equal those of $A_{ii}$.
Thus, $A_{ii}$ is
the compartmental matrix of $\mathcal{M}'$.
\end{proof}

The following terminology matches that of Buslov~\cite{buslov}:
\begin{definition} \label{def:forest}
Let $G$ be a (directed) graph.
\begin{enumerate} 
	\item A {\em spanning subgraph} of $G$ is a subgraph of $G$ with the same set of vertices as $G$.
	\item An {\em incoming forest} is a directed graph such that (a) the underlying undirected graph has no cycles and (b) each node has at most one outgoing edge.
	\item For an incoming forest $F$, let $\pi_F$ denote the product of the labels of all edges in the forest, that is, $\pi_F = \prod_{(i,j) \in E(F)} a_{ji}$, where $a_{ji}$ labels the edge $i \to j$.
	\item  Let $ \mathcal{F}_k( G) $ denote the set of all $k$-edge, spanning, incoming forests of $G$.
\end{enumerate}
\end{definition}

\subsection{A formula for the coefficient map} \label{sec:coef}
Our formula for the coefficient map expresses each coefficient as a sum, over certain spanning forests, of the product of the edge labels in the forest (Theorem~\ref{thm:coeff-i-o-general}).  
The formula is an ``expanded out'' version of a result of Meshkat and Sullivant
 \cite[Theorem 3.2]{MeshkatSullivant} 
 that showed the coefficient map factors through the cycles in the leak-augmented graph.
The difference is due to the fact that Meshkat and Sullivant treated diagonal entries of $A$ as separate variables (e.g., $a_{ii}$), while our diagonal entries are negative sums of a leak and/or rates (e.g., $-a_{0i}-a_{2i}-a_{3i}$).

\begin{theorem}[Coefficients of input-output equations] \label{thm:coeff-i-o-general}
 \label{thm:coeff-i-o-general}
Consider a linear compartmental model $\mathcal{M} = (G, In, Out, Leak)$
 that
	has an input and output in compartment 1 (and no other inputs or outputs).
Let $n$ denote the number of compartments, and $A$ the compartmental matrix.  Write the input-output equation~\eqref{eq:i-o-general} as:
	\begin{align} \label{eq:coeff-i-o-general}
	 y_1^{(n)} + c_{n-1}  y_1^{(n-1)} + \dots  + c_1 y_1' + c_0y_1 ~=~  u_1^{(n-1)} + d_{n-2} u_1^{(n-2)} + \dots  + d_1 u_1' + d_0 u_1~.
	\end{align}
	Then the coefficients of this input-output equation are as follows:
\begin{align*}
	c_i ~&=~
			 \sum_{F \in \mathcal{F}_{n-i}( \widetilde G)} \pi_F   \quad \quad \quad \text{for } i=0,1,\dots,n-1~, \quad  \text{and} \\
	d_i ~&=~ 
			 \sum_{F \in \mathcal{F}_{n-i-1}( \widetilde G_1)} \pi_F    \quad \quad\quad \text{for } i=0,1,\dots,n-2~.
\end{align*}
%
%
%
%
\end{theorem}

\begin{remark} \label{rmk:BGMSS}
Theorem~\ref{thm:coeff-i-o-general} was extended recently by Bortner {\em et al.}\ to allow for models with input and output in distinct compartments~\cite{BGMSS}.
\end{remark}


The proof of Theorem~\ref{thm:coeff-i-o-general} requires the following result, which interprets the coefficients of the characteristic polynomial of a compartmental matrix. 
\begin{proposition} \label{prop:coeff}
Let $A$ be the compartmental matrix
of a linear compartmental model with $n$ compartments and leak-augmented graph $\widetilde G$.  
Write the characteristic polynomial of $A$ as:
	\begin{align*}
	\det (\lambda I - A) ~=~ \lambda^n + e_{n-1} \lambda^{n-1} + \cdots + e_0~.
	\end{align*}
Then $e_i$ (for $i=0,1,\dots,n-1$) is the sum over $(n-i)$-edge, spanning, incoming forests of
 $\widetilde G$, where each summand is the product of the edge labels in the forest:
	\begin{align*}
		e_i ~&=~ 
			\sum_{F \in \mathcal{F}_{n-i}( \widetilde G)} \pi_F   ~.
	\end{align*}
\end{proposition}

In the Appendix,
we prove Proposition~\ref{prop:coeff} and explain how it is related to similar results.

\begin{proof}[Proof of Theorem~\ref{thm:coeff-i-o-general}]
By Proposition~\ref{prop:MSE}, the coefficient $c_i$ of $y^{(i)}$ 
in the input-output equation \eqref{eq:coeff-i-o-general} is the coefficient of $\lambda^i$ in the characteristic polynomial $\det  (\lambda I - A)$ of the compartmental matrix $A$.  Hence, the desired result follows immediately from Proposition~\ref{prop:coeff}.  

Now consider the right-hand side of the input-output equation~\eqref{eq:coeff-i-o-general}. 
Let $A_{11}$ denote the matrix obtained from $A$ by removing row 1 and column 1.
By Lemma~\ref{lem:remove-1}, $A_{11}$ is the compartmental matrix for any model with leak-augmented graph $\widetilde G_1$.  
So, by Proposition~\ref{prop:coeff}, the sum 
$  \sum_{F \in \mathcal{F}_{n-i-1}( \widetilde G_1)} \pi_F $
equals the coefficient of $\lambda^i$ in the characteristic polynomial $\det  (\lambda I - A_{11})= \det  (\lambda I - A)_{11}$ (where the first identity matrix $I$ has size $n$ and the second has size $n-1$).  
This coefficient, by Proposition~\ref{prop:MSE}, equals $d_i$, and this completes the proof.
\end{proof}


\begin{remark}[Jacobian matrix of the coefficient map] \label{rmk:jac-mat}
In the setting of Theorem~\ref{thm:coeff-i-o-general},
each coefficient $c_k$ of the input-output equation
is the sum of products of edge labels of a forest, and thus is multilinear in the parameters $a_{ij}$.
Therefore, in the row of the Jacobian matrix corresponding to $c_k$, the entry in the column corresponding to some $a_{lm}$
is obtained from $c_k$ by setting $a_{lm}$=1 in those terms divisible by $a_{lm}$ and then setting all other terms to 0.
\end{remark}

\section{The singular locus: mammillary, catenary, and cycle models} \label{sec:singular}
In this section, we establish the singular-locus equations for the mammilary (star) and cycle models, which were displayed in Table~\ref{tab:summary} (Theorems~\ref{thm:mam} and~\ref{thm:cycle}).  We also state our conjecture for the singular-locus equation for the catenary (path) model (Conjecture~\ref{conj:cat}). We additionally pose a related conjecture for models that are formed by bidirectional trees, which include the catenary model (Conjecture~\ref{conj:tree}).

\subsection{Mammillary (star) models} \label{sec:mam}
\begin{theorem}[Mammillary] \label{thm:mam}
Assume $n \geq 2$.
The $n$-compartment mammillary (star) model in Figure~\ref{fig:cycle}
is generically locally identifiable, and the equation of the singular locus is:
	\begin{align} \label{eq:mam}
	(a_{12} a_{13} \dots a_{1,n} ) 
	\prod_{2 \leq i < j \leq n} \left(a_{1i} - a_{1j} \right)^2~.
	\end{align}
\end{theorem}

\begin{proof}
The compartmental matrix for this model is 
	\begin{align*} 
	A ~=~
	\begin{pmatrix}
	-a_{01}-(a_{21}+ \dots +a_{n1}) & a_{12} & a_{13} & \dots & a_{n1} \\ 
	a_{21} & - a_{21} & 0 & \dots & 0 \\
	a_{31} & 0 & -a_{13} & & 0 \\
	\vdots & \vdots  & & \ddots & \\
	a_{n1} & 0 & 0 &  & -a_{1n}  
	\end{pmatrix}~.
	\end{align*}
Let $E_{j} (x_1,\dots, x_m )$ denote the $j$-th elementary symmetric polynomial on 
$x_1,x_2, \dots, x_m$; and let
$E_{j} (\hat a_{1k} )$ denote the $j$-th elementary symmetric polynomial on 
$a_{12}, \dots, a_{1,k-1}, a_{1,k+1}, \dots, a_{1n}$. 
Then, the coefficients on the left-hand side of the input-output equation~\eqref{eq:i-o-general} are, by Theorem~\ref{thm:coeff-i-o-general}, the following:
	\begin{align*}
	c_i ~&=~ a_{01} E_{n-i-1}(a_{12}, \dots, a_{1n}) + 
		\left(
		a_{21}E_{n-i} (\hat a_{12} ) + 
		a_{31}E_{n-i} (\hat a_{13} ) + \dots  +
		a_{n1}E_{n-i} (\hat a_{1n} ) 		
		\right)  \\
		& \quad \quad
		+
		E_{n-i}(a_{12}, \dots, a_{1n}) ~ 
	\end{align*}
for $i=0,1,\dots, n-1$.   
As for the coefficients of the right-hand side of the input-output equation, they are as follows, by Proposition~\ref{prop:MSE}:
	\begin{align*}
	d_i ~=~  E_{n-i-1}(a_{12}, \dots, a_{1n})~ \quad \quad \text{for $i=0,1,\dots, n-2$~.	}
	\end{align*}
Consider the coefficient map $(c_{n-1},c_{n-2}, \dots, c_0, ~d_{n-2}, d_{n-3}, \dots, d_0)$.  Its Jacobian matrix, 
where the order of variables is $(a_{21}, a_{31}, \dots, a_{n1}, ~a_{01}, ~a_{12}, a_{13}, \dots, a_{1n})$,
 has the following form:
\[
\left( 
\begin{array}{c@{}c@{}c}
 \left[\begin{array}{c}
         M \\      \end{array}\right] & \star & \star \\     
         \mathbf{0} & \left[\begin{array}{c}
                       a_{12} a_{13}\cdots a_{1n}   \\                          \end{array}\right] & \star \\   
                        \mathbf{0} & \mathbf{0} & \left[ \begin{array}{c}
                                   M\\                                      \end{array}\right] \    \end{array}\right)~,
\]
where $M$ is the following $(n-1) \times (n-1)$ matrix:
	\begin{align*} 
	M ~=~
	\begin{pmatrix}
	1 & 1& \dots & 1 \\
	E_1(\hat a_{12}) & E_1(\hat a_{13}) & \dots & E_1(\hat a_{1n}) \\
	E_2(\hat a_{12}) & E_2(\hat a_{13}) & \dots & E_2(\hat a_{1n}) \\
	\vdots & \vdots &  & \vdots \\
	E_{n-2} (\hat a_{12}) & E_{n-2} (\hat a_{13}) & \dots & E_{n-2} (\hat a_{1n}) 	
	\end{pmatrix}~.
	\end{align*}
Thus, to prove the desired formula~\eqref{eq:mam}, we need only show that the
determinant of $M$ equals, up to sign, the {\em Vandermonde polynomial} on 
$(a_{12}, \dots, a_{1n})$:
	\begin{align} \label{eq:Van}
	\det M ~=~  \pm \prod_{2 \leq i < j \leq n} (a_{1i} - a_{1j})~.
	\end{align}	
To see this, note first that both polynomials have the same multidegree: the degree with respect to the $a_{1j}$'s of $\det M$ is $0+1+ \dots + (n-2)$ (because the entries in row-$i$ of $M$ have degree $i-1$), which equals ${n-1 \choose 2}$, and this is the degree of the Vandermonde polynomial on the right-hand side of equation~\eqref{eq:Van}.  Also, note that both polynomials are, up to sign, monic.  

So, to prove the claimed equality~\eqref{eq:Van}, it suffices to show that when $2 \leq i < j \leq n$, the term
$(a_{1i}-a_{1j})$ divides $\det M$.
Indeed, when $a_{1i}=a_{1j}$, then the columns of $M$ that correspond to $a_{i1}$ and $a_{1j}$ (namely, the $(i-1)$-st and 
$(j-1)$-st
columns) coincide, and thus $\det M = 0$.  Hence, $(a_{1i}-a_{1j}) | \det M$ (by the Nullstellensatz).
\end{proof}

\begin{corollary}
Let $n \geq 2$,
 and let $\mathcal M$ be the $n$-compartment mammillary (star) model in Figure~\ref{fig:cycle}.  If all rate parameters $ a_{01}, a_{12}, a_{21}, a_{13}, a_{31}, \dots a_{1n}, a_{n1}$ are positive and unique, then the coefficient map of $\mathcal M$ is locally one-to-one around the parameter point and thus the parameters can be recovered (up to a finite set) from input-output data. 
\end{corollary}

\subsection{Cycle models} \label{sec:cycle}

\begin{theorem}[Cycle] \label{thm:cycle}
Assume $n \geq 3$.
The $n$-compartment cycle model in Figure~\ref{fig:cycle}
is generically locally identifiable, and the equation of the singular locus is:
	\begin{align*}
a_{32}a_{43} \dots a_{n,n-1} a_{1,n}
	\prod_{2 \leq i < j \leq n}
		\left( a_{i+1,i} - a_{j+1,j}
			\right)~.
	\end{align*}
\end{theorem}

\begin{proof}
The compartmental matrix for this model is 
	\begin{align*} 
	A ~=~
	\begin{pmatrix}
	-a_{01}-a_{21} & 0 & 0 & \dots &0 & a_{1n} \\ 
	a_{21} & - a_{32} & 0 & \dots & 0 & 0 \\
	0 & a_{32} & -a_{43} & & & 0 \\
	\vdots &    & \ddots & \ddots & & \vdots \\
	0 &  & &a_{n-1,n-2}&-a_{n,n-1}& 0\\
	0 & 0 & \dots &0 & a_{n,n-1} & -a_{1n}  
	\end{pmatrix}~.
	\end{align*}

Let $E_{j}:=E (a_{32},a_{43}, \dots,a_{1,n} )$ denote the $j$-th elementary symmetric polynomial on 
\\ $a_{32}, a_{43},\dots,a_{1,n}$; and let
$E_{j} (\hat a_{k+1,k} )$ denote the $j$-th elementary symmetric polynomial on 
$a_{32}, a_{43}, \dots,a_{k,k-1}$, $a_{k+2,k+1}, \dots, a_{1,n}$, where  $E_{j} (\hat a_{n+1,n} ):=E_j (\hat a_{1,n})$.
Then, the coefficients on the left-hand side of the input-output equation~\eqref{eq:coeff-i-o-general} are, by Theorem~\ref{thm:coeff-i-o-general}, the following:
	\begin{align} \label{eq:coef-cycle-1}
	c_0 ~&=~ a_{01} E_{n-1}~, \quad \quad {\rm and} \\
	c_i ~&=~ \left(a_{01}+a_{21}\right) E_{n-i-1}  +  E_{n-i-2}
			\quad \quad \text{(for } i = 1,2,\dots, n-1\text{)}~.
	\end{align}
As for the coefficients of the right-hand side of the input-output equation, they are as follows, by Proposition~\ref{prop:MSE}:
	\begin{align}\label{eq:coef-cycle-2}
	d_i ~=~  E_{n-i-1}~ \quad \quad \text{(for $i=0,1,\dots, n-2$).	}
	\end{align}
	
Consider the coefficient map $(c_0, c_1, \dots, c_{n-1}, ~ d_0,d_1, \dots, d_{n-2})$.  Its Jacobian matrix, 
where the order of variables is $(a_{01},a_{21}, a_{32}, \dots,  a_{1n})$, has
the following form:
\begin{equation} \label{eq:jac-cycle}
J~=~
\left( 
\begin{array}{c@{}c}
 \left[\begin{array}{cc}
         E_{n-1} & 0 \\
         E_{n-2} & E_{n-2} \\
	\vdots & \vdots \\
         E_{1} & E_{1} \\
         1& 1
          \\      \end{array}\right] & \star  \\     
                        \mathbf{0} &  \left[ \begin{array}{c}
                                   M\\                                      \end{array}\right] \    \end{array}\right)~,
\end{equation}
where $M$ is the following $(n-1) \times (n-1)$ matrix:
	\begin{align*} 
	M ~=~
	\begin{pmatrix}
	E_{n-2} (\hat a_{32}) & E_{n-2} (\hat a_{43}) & \dots & E_{n-2} (\hat a_{1n}) 	\\
	E_{n-3} (\hat a_{32}) & E_{n-3} (\hat a_{43}) & \dots & E_{n-3} (\hat a_{1n}) 	\\
	\vdots & \vdots &  & \vdots \\
	E_{1} (\hat a_{32}) & E_{1} (\hat a_{43}) & \dots & E_{1} (\hat a_{1n}) 	\\
	1 & 1& \dots & 1 \\
	\end{pmatrix}~.
	\end{align*}
In the upper-left $(n \times 2)$-block of the matrix $J$ in equation~\eqref{eq:jac-cycle}, rows 2 through $(n-1)$ are scalar multiples of the bottom row of 1's.  
Thus, if we let $\widetilde J$ denote the square matrix (of size $n+1$) obtained by removing from $J$ rows 2 through $(n-1)$, then the singular-locus equation of the model is $\det \widetilde J$.  
Indeed, all nonzero $(n+1)\times (n+1)$ minors of $J$ are scalar multiples of $\det \widetilde J$, and thus the singular locus is defined by the single equation $\det \widetilde J=0$.

From equality~\eqref{eq:Van} in the proof of Theorem~\ref{thm:mam}, we know $\det M= \pm 
\prod_{2 \leq i < j \leq n}
		\left( a_{i+1,i} - a_{j+1,j} \right)$.
	 	Thus, the equation of the singular locus is, up to sign, as follows:
\[
\det \widetilde J ~=~ (E_{n-1}) (\det M) ~=~ \pm a_{32}a_{43} \dots a_{n,n-1} a_{1,n}
	\prod_{2 \leq i < j \leq n}
		\left( a_{i+1,i} - a_{j+1,j}
			\right)~.
			\]%
\end{proof}


\begin{corollary}
Let $n \geq 3$, and let $\mathcal M$ be the $n$-compartment cycle model in Figure~\ref{fig:cycle}.  If all rate parameters $a_{01}, a_{21}, a_{32},a_{43}, \dots,a_{1,n}$ are positive and unique, then the coefficient map of $\mathcal M$ is locally one-to-one around the parameter point and the parameters 
can be recovered (up to a finite set) from input-output data. 
\end{corollary}

\subsection{Catenary (path) models} \label{sec:cat}
\begin{conjecture} \label{conj:cat}
Assume $n \geq 2$.
For the  $n$-compartment catenary (path) model in Figure~\ref{fig:cat},
the equation of the singular locus is:
	\begin{align} \label{eq:cat}
	a_{12}^{n-1} 
	(a_{21} a_{23})^{n-2} 
	(a_{32} a_{34})^{n-3}
	\dots
	(a_{n-1,n-2} a_{n-1,n})~.
	\end{align}
\end{conjecture}

\begin{remark}
The structure of the conjectured equation~\eqref{eq:cat} suggests a proof by induction on $n$, but we currently do not know how to complete, for such a proof, the inductive step.
\end{remark}

%
We can prove the following weaker version of Conjecture~\ref{conj:cat}:
\begin{proposition} \label{prop:cat}
For the  $n$-compartment catenary (path) model in Figure~\ref{fig:cat}, the following parameters divide
the equation of the singular locus:
	\begin{align*}
	a_{21} \quad \quad {\rm and} \quad \quad
	a_{12}, ~a_{23}, ~...~,
	~ a_{n-1,n}~.
	\end{align*}
\end{proposition}

\begin{proof}
By Theorem~\ref{thm:coeff-i-o-general}, the coefficients of the input-output equation of the catenary model arise from spanning forests of the following graphs:
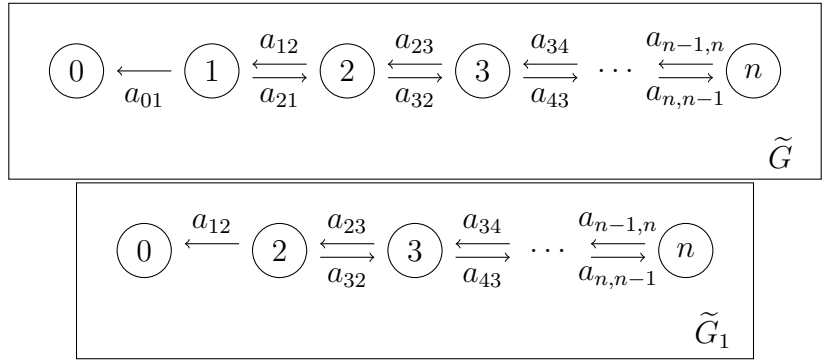
\begin{center}
	\begin{tikzpicture}[scale=1.8]
 	\draw (-1,0) circle (0.2);	
 	\draw (0,0) circle (0.2);	
 	\draw (1,0) circle (0.2);	
 	\draw (2,0) circle (0.2);	
 	\draw (4,0) circle (0.2);	
    	\node[] at (-1, 0) {0};
    	\node[] at (0, 0) {1};
    	\node[] at (1, 0) {$2$};
    	\node[] at (2, 0) {$3$};
    	\node[] at (3, 0) {$\dots$};
    	\node[] at (4, 0) {$n$};
	 \draw[<-] (0.3, 0.05) -- (0.7, 0.05);	
	 \draw[->] (0.3, -0.05) -- (0.7, -0.05);	
	 \draw[<-] (1.3, 0.05) -- (1.7, 0.05);	
	 \draw[->] (1.3, -0.05) -- (1.7, -0.05);	
	 \draw[<-] (2.3, 0.05) -- (2.7, 0.05);	
	 \draw[->] (2.3, -0.05) -- (2.7, -0.05);	
	 \draw[<-] (3.3, 0.05) -- (3.7, 0.05);	
	 \draw[->] (3.3, -0.05) -- (3.7, -0.05);	
   	 \node[] at (0.5, 0.2) {$a_{12}$};
   	 \node[] at (1.5, 0.2) {$a_{23}$};
   	 \node[] at (2.5, 0.2) {$a_{34}$};
   	 \node[] at (3.5, 0.2) {$a_{n-1,n}$};
   	 \node[] at (0.5, -0.2) {$a_{21}$};
   	 \node[] at (1.5, -0.2) {$a_{32}$};
   	 \node[] at (2.5, -0.2) {$a_{43}$};
   	 \node[] at (3.5, -0.2) {$a_{n,n-1}$};
	 \draw[<-] (-0.7, 0) -- (-0.3, 0);	
   	 \node[] at (-0.5, -0.2) {$a_{01}$};
    	\node[above] at (4.2, -0.8) {$\widetilde{G}$};
\draw (-1.5,-.8) rectangle (4.5, .5);
	\end{tikzpicture}
	\begin{tikzpicture}[scale=1.8]
 	\draw (0,0) circle (0.2);	
 	\draw (1,0) circle (0.2);	
 	\draw (2,0) circle (0.2);	
 	\draw (4,0) circle (0.2);	
    	\node[] at (0, 0) {0};
    	\node[] at (1, 0) {$2$};
    	\node[] at (2, 0) {$3$};
    	\node[] at (3, 0) {$\dots$};
    	\node[] at (4, 0) {$n$};
	 \draw[<-] (0.3, 0.05) -- (0.7, 0.05);	
	 \draw[<-] (1.3, 0.05) -- (1.7, 0.05);	
	 \draw[->] (1.3, -0.05) -- (1.7, -0.05);	
	 \draw[<-] (2.3, 0.05) -- (2.7, 0.05);	
	 \draw[->] (2.3, -0.05) -- (2.7, -0.05);	
	 \draw[<-] (3.3, 0.05) -- (3.7, 0.05);	
	 \draw[->] (3.3, -0.05) -- (3.7, -0.05);	
   	 \node[] at (0.5, 0.2) {$a_{12}$};
   	 \node[] at (1.5, 0.2) {$a_{23}$};
   	 \node[] at (2.5, 0.2) {$a_{34}$};
   	 \node[] at (3.5, 0.2) {$a_{n-1,n}$};
   	 \node[] at (1.5, -0.2) {$a_{32}$};
   	 \node[] at (2.5, -0.2) {$a_{43}$};
   	 \node[] at (3.5, -0.2) {$a_{n,n-1}$};
    	\node[above] at (4.2, -0.8) {$\widetilde{G}_1$};
\draw (-0.5,-.8) rectangle (4.5, .5);
	\end{tikzpicture}
\end{center}
More specifically, some of the coefficients are as follows:
\begin{alignat}{2} \label{eq:coeffs-pf}
c_{0} & = (a_{01}) a_{12}a_{23} \dots a_{n-1,n}   & \mathrm{ (corresponds~to~Row~1)} \notag
\\
c_1 &= 	\sum_{F \in \mathcal{F}_{n-1}( \widetilde G)} \pi_F
  &  \mathrm{ (Row~2)}
\notag
\\
c_{n-1} &=  (a_{01}+a_{21}) + a_{12} +a_{23}  +a_{32} + \dots +a_{n-1,n} + a_{n,n-1}
  &  \mathrm{ (Row~} n\mathrm{)}
\\
d_0 &= a_{12}a_{23} \dots a_{n-1,n}    & \mathrm{ (Row~} n+1\mathrm{)}
\notag
\\
d_1 &= \sum_{F \in \mathcal{F}_{n-1}( \widetilde G_1)} \pi_F ~=
\sum_{F \in \mathcal{F}_{n-1}( \widetilde G) ~:~ F \mathrm{~does~not~involve~} a_{01} {\rm ~or~}  a_{21}} \pi_F
  &  \mathrm{ (Row~} n+2\mathrm{)}
\notag
\\
d_{n-2} &=  a_{12} +a_{23}  +a_{32} + \dots +a_{n-1,n} + a_{n,n-1}  & \mathrm{ (Row~} 2n-1\mathrm{)}
\notag
\end{alignat}
For each coefficient in~\eqref{eq:coeffs-pf}, we indicated the corresponding row of the 
$(2n-1) \times (2n-1)$ Jacobian matrix for the coefficient map $(c_0,~c_1,~\dots,~c_{n-1}, ~d_{0},~d_1,~\dots,~d_{n-2})$.

Perform the following elementary row operations (which do not affect the determinant) on the Jacobian matrix:
\begin{enumerate}[(i)]
\item Row $1$ := Row $1$ -  $a_{01}$Row $(n+1)$  
\item	 Row $2$ := Row $2$ -  $a_{01}$Row $(n+2)$ - Row $(n + 1)$  
\item  Row $n$ :=  Row $n$ - Row $(2n-1)$.
\end{enumerate}
Next, we reorder the columns so that the first four columns are indexed by $a_{01}, ~a_{n,n-1}, ~a_{21}, ~a_{12}$.  We claim that in the resulting matrix, the submatrix formed by Rows 1, 2, and $n$ has the following form:
\begin{align} \label{eq:matrix-star}
 \left( \begin{array}{cccc|ccc}
a_{12}a_{23}a_{34}...a_{n-1,n} & 0 & 0 & 0 & 0 & \ldots & 0 \\
* & a_{21}a_{32}...a_{n-1,n-2} & * & 0 & \bigstar & \ldots & \bigstar \\
1 & 0 & 1 & 0 & 0 & \ldots & 0\\
 \end{array} \right)
 \end{align}
The forms of the first and third rows follow from~\eqref{eq:coeffs-pf} and the row operations (i) and (iii).  As for the second row of the matrix~\eqref{eq:matrix-star}, 
consider an entry that is {\em not} labeled by $*$, i.e., an entry in a column indexed by some $a_{ij}$ with $a_{ij} \neq a_{01}, a_{21}$.  This entry, via a straightforward argument using~\eqref{eq:coeffs-pf} and the row operation (iii), is the following sum over the $(n-1)$-edge incoming forests of $\widetilde{G}$ that involve both edges $a_{21}$ and $a_{ij}$:
 \[
 \sum_{H \sqcup \{a_{21}, a_{ij} \} \in \mathcal{F}_{n-i}( \widetilde G)} a_{21} \pi_H   ~.
 \]	
(Here $\sqcup$ denotes disjoint union.)  Each forest in such a sum has the following form, for some $k=2,3,\dots, n$:

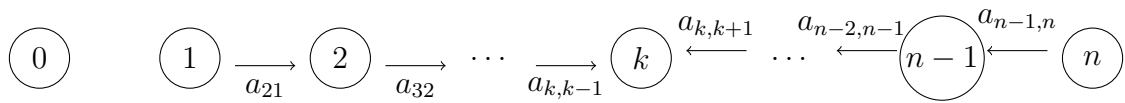
\begin{center}
	\begin{tikzpicture}[scale=2]
 	\draw (-1,0) circle (0.2);	
 	\draw (0,0) circle (0.2);	
 	\draw (1,0) circle (0.2);	
 	\draw (3,0) circle (0.2);	
 	\draw (5,0) circle (0.28);	
 	\draw (6,0) circle (0.2);	
    	\node[] at (-1, 0) {0};
    	\node[] at (0, 0) {1};
    	\node[] at (1, 0) {$2$};
    	\node[] at (2, 0) {$\dots$};
    	\node[] at (3, 0) {$k$};
    	\node[] at (4, 0) {$\dots$};
    	\node[] at (5, 0) {$n-1$};
    	\node[] at (6, 0) {$n$};
	 \draw[->] (0.3, -0.05) -- (0.7, -0.05);	
	 \draw[->] (1.3, -0.05) -- (1.7, -0.05);	
	 \draw[->] (2.3, -0.05) -- (2.7, -0.05);	
	 \draw[<-] (3.3, 0.05) -- (3.7, 0.05);	
	 \draw[<-] (4.3, 0.05) -- (4.7, 0.05);	
	 \draw[<-] (5.3, 0.05) -- (5.7, 0.05);	
   	 \node[] at (3.5, 0.2) {$a_{k,k+1}$};
   	 \node[] at (4.4, 0.2) {$a_{n-2,n-1}$};
   	 \node[] at (5.5, 0.25) {$a_{n-1,n}$};
   	 \node[] at (0.5, -0.2) {$a_{21}$};
   	 \node[] at (1.5, -0.2) {$a_{32}$};
   	 \node[] at (2.5, -0.2) {$a_{k,k-1}$};
	\end{tikzpicture}
\end{center}

The only such forest involving the edge $a_{n,n-1}$ is the $k=n$ case, so the $(2,2)$-entry in matrix~\eqref{eq:matrix-star} is indeed $a_{21}a_{32}...a_{n-1,n-2}$.  Also, note that $a_{21}$ divides all 
entries labeled by $\bigstar$.

Next, it is straightforward to row-reduce the matrix~\eqref{eq:matrix-star}, without affecting the determinant or the values of the entries labeled by $\bigstar$, to obtain:
\[ \left( \begin{array}{cccc|ccc}
a_{12}a_{23}a_{34}...a_{n-1,n} & 0 & 0 & 0 & 0 & \ldots & 0 \\
0 & a_{21}a_{32}...a_{n-1,n-2} & 0 & 0 & \bigstar & \ldots & \bigstar \\
0 & 0 & 1 & 0 & 0 & \ldots & 0\\
 \end{array} \right)\] 
Thus, from examining Rows 1 and 2, and the fact that all $\bigstar$-labeled entries are multiples of $a_{21}$, we conclude that $a_{21}$ and $a_{12}a_{23}a_{34}...a_{n-1,n}$ both divide 
the determinant of the full Jacobian matrix. This determinant is the singular-locus equation, so we are done.
\end{proof}

\begin{rmk}  
Conjecture \ref{conj:cat} 
asserts that all parameters except $a_{n,n-1}$ and (the leak) $a_{01}$ divide the singular-locus equation of the catenary model.  
For some of these parameters, this assertion was proven in Proposition \ref{prop:cat}.  For the remaining parameters, namely, 
$a_{32}, a_{43}, \ldots, a_{n-1,n-2}$ (which form a path from compartment $2$ to compartment $n-1$), we hope to prove in the future that they too divide the singular-locus equation.
\end{rmk}

\subsection{Tree conjecture}
In this subsection, 
we generalize Conjecture~\ref{conj:cat}, which pertained to catenary (path) models,
 to ``tree models'' (Conjecture~\ref{conj:tree}).
To motivate the new conjecture, we begin by revisiting the 4-compartment catenary and mammillary models.  
We depict these models in Figure~\ref{fig:multiplicity}, where 
instead of labeling the edge $(i,j)$ with $a_{ji}$, we label the edge with the
{\em multiplicity} of $a_{ji}$ in the equation of the singular locus, 
i.e., the largest $p$ such that $a_{ji}^p$ divides the equation (recall Conjecture~\ref{conj:cat} and Theorem~\ref{thm:mam}).

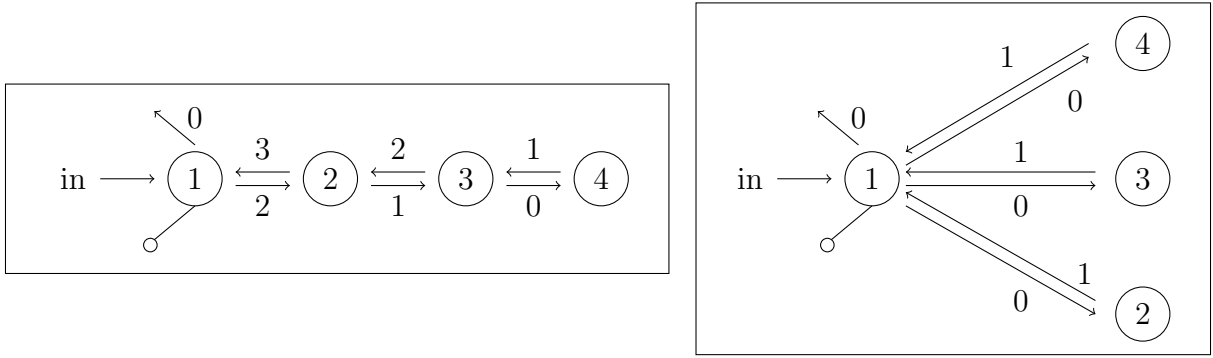
\begin{figure}[ht]
\begin{center}
	\begin{tikzpicture}[scale=1.8]
 	\draw (0,0) circle (0.2);	
 	\draw (1,0) circle (0.2);	
 	\draw (2,0) circle (0.2);	
 	\draw (3,0) circle (0.2);	
    	\node[] at (0, 0) {1};
    	\node[] at (1, 0) {$2$};
    	\node[] at (2, 0) {$3$};
    	\node[] at (3, 0) {$4$};
	 \draw[<-] (0.3, 0.05) -- (0.7, 0.05);	
	 \draw[->] (0.3, -0.05) -- (0.7, -0.05);	
	 \draw[<-] (1.3, 0.05) -- (1.7, 0.05);	
	 \draw[->] (1.3, -0.05) -- (1.7, -0.05);	
	 \draw[<-] (2.3, 0.05) -- (2.7, 0.05);	
	 \draw[->] (2.3, -0.05) -- (2.7, -0.05);	
   	 \node[] at (0.5, 0.22) {3}; 
   	 \node[] at (1.5, 0.22) {2}; 
   	 \node[] at (2.5, 0.22) {1}; 
   	 \node[] at (0.5, -0.2) {2}; 
   	 \node[] at (1.5, -0.2) {1}; 
   	 \node[] at (2.5, -0.2) {0}; 
 	\draw (-0.33,-.49) circle (0.05);	
	 \draw[-] (0, -.2 ) -- (-0.3, -.45);	
	 \draw[->] (-0.7, 0) -- (-0.3, 0);	
   	 \node[] at (-.9, 0) {in};
	 \draw[->] (0, .25) -- (-0.3, .5);	
   	 \node[] at (0, 0.45){0}; 
\draw (-1.4,-.7) rectangle (3.5, .7);
%
 	\draw (5,0) circle (0.2);
 	\draw (7,-1) circle (0.2);
 	\draw (7, 0) circle (0.2);
 	\draw (7,1) circle (0.2);
    	\node[] at (5, 0) {1};
    	\node[] at (7, -1) {2};
    	\node[] at (7, 0) {3};
    	\node[] at (7, 1) {$4$};
	 \draw[->] (5.25, -.2) -- (6.65, -1);
	 \draw[<-] (5.25, -.1) -- (6.65, -0.9);
	 \draw[->] (5.25, -0.05) -- (6.65, -0.05);
	 \draw[<-] (5.25, 0.05) -- (6.65, 0.05);
	 \draw[->] (5.25, 0.1) -- (6.6, 0.9);
	 \draw[<-] (5.25, 0.2) -- (6.6, 1);
   	 \node[] at (6.1, -0.9) {0}; 
   	 \node[] at (6.57, -0.7){1}; 
   	 \node[] at (6.1, 0.2){1}; 
   	 \node[] at (6.1, -0.2){0}; 
   	 \node[] at (6, 0.9){1}; 
   	 \node[] at (6.5, 0.58){0}; 
 	\draw (4.67,-.49) circle (0.05);	
	 \draw[-] (5, -.2 ) -- (4.7, -.45);	
	 \draw[->] (4.3, 0) -- (4.7, 0);	
   	 \node[] at (4.1, 0) {in};
	 \draw[->] (4.9, .25) -- (4.6, .5);	
   	 \node[] at (4.9, 0.45) {0}; 
\draw (3.7,-1.3) rectangle (7.5, 1.3);
	\end{tikzpicture}
\end{center}
\caption{A catenary and a mammillary model, with edges $(i,j)$ labeled by the multiplicity of $a_{ji}$ in the  corresponding singular-locus equation.}
\label{fig:multiplicity}
\end{figure}

Now consider the model in Figure~\ref{fig:multiplicity-2}, which also has edges labeled by multiplicities.
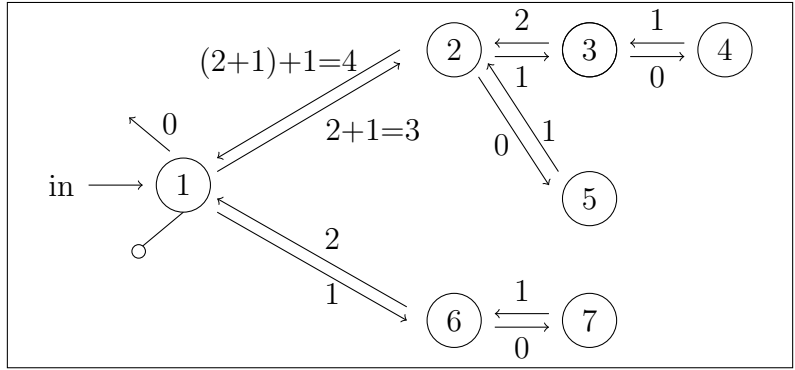
\begin{figure}[ht]
\begin{center}
	\begin{tikzpicture}[scale=1.8]
 	\draw (3,0) circle (0.2);
 	\draw (5,1) circle (0.2);
 	\draw (6,1) circle (0.2);
 	\draw (5,-1) circle (0.2);
 	\draw (6,-1) circle (0.2);
 	\draw (6,-0.1) circle (0.2);
 	\draw (6,1) circle (0.2);
 	\draw (7,1) circle (0.2);
	 \draw[->] (3.25, -.2) -- (4.65, -1);
	 \draw[<-] (3.25, -.1) -- (4.65, -0.9);
	 \draw[->] (3.25, 0.1) -- (4.6, 0.9);
	 \draw[<-] (3.25, 0.2) -- (4.6, 1);
	 \draw[<-] (5.25, .90) -- (5.77, 0.1);	
	 \draw[->] (5.18, .80) -- (5.70, 0);	
	 \draw[<-] (5.3, 1.05) -- (5.7, 1.05);	
	 \draw[->] (5.3, 0.95) -- (5.7, 0.95);	
	 \draw[->] (6.3, 0.95) -- (6.7, 0.95);	
	 \draw[<-] (6.3, 1.05) -- (6.7, 1.05);
	 \draw[<-] (5.3, -0.95) -- (5.7, -0.95);	
	 \draw[->] (5.3, -1.05) -- (5.7, -1.05);	
    	\node[] at (3, 0) {1};
    	\node[] at (5, 1) {2};
    	\node[] at (6, 1) {3};
    	\node[] at (7, 1) {4};
    	\node[] at (6, -0.1) {5};
    	\node[] at (5, -1) {6};
    	\node[] at (6, -1) {7};
   	 \node[] at (4.1, -0.8) {1}; 
   	 \node[] at (4.1, -0.4){2}; 
   	 \node[] at (5.5, -0.78) {1}; 
   	 \node[] at (5.5, -1.2) {0}; 
   	 \node[] at (3.7, 0.9){(2+1)+1=4}; 
   	 \node[] at (4.4, 0.4){2+1=3}; 
   	 \node[] at (5.5, 1.22) {2}; 
   	 \node[] at (5.5, 0.8) {1}; 
   	 \node[] at (6.5, 1.22) {1}; 
   	 \node[] at (6.5, 0.8) {0}; 
   	 \node[] at (5.7, 0.4) {1}; 
   	 \node[] at (5.35, 0.3) {0}; 
\draw (1.7,-1.35) rectangle (7.5, 1.35);
 	\draw (2.67,-.49) circle (0.05);	
	 \draw[-] (3, -.2 ) -- (2.7, -.45);	
	 \draw[->] (2.3, 0) -- (2.7, 0);	
   	 \node[] at (2.1, 0) {in};
	 \draw[->] (2.9, .25) -- (2.6, .5);	
   	 \node[] at (2.9, 0.45){0}; 
	\end{tikzpicture}
\end{center}
\caption{A linear compartmental model, with edges $(i,j)$ labeled by the multiplicity of $a_{ji}$ in the corresponding singular-locus equation.}
\label{fig:multiplicity-2}
\end{figure}

Notice that all leaf-edges
in Figures~\ref{fig:multiplicity} and~\ref{fig:multiplicity-2}
 have the following labels:
\begin{center}
	\begin{tikzpicture}[scale=1.6]
 	\draw (0,0) circle (0.2);	
 	\draw (1,0) circle (0.2);	
	 \draw[<-] (0.3, 0.05) -- (0.7, 0.05);	
	 \draw[->] (0.3, -0.05) -- (0.7, -0.05);	
   	 \node[] at (0.5, 0.22) {1}; 
   	 \node[] at (0.5, -0.2) {0}; 
	\end{tikzpicture}
\end{center}
Also, as edges move one step closer to compartment 1, the corresponding edge labels increase by 1, 
except at compartment 2 in Figure~\ref{fig:multiplicity-2}.
Incident to that compartment are the edges (2,1) and (1,2), whose labels are written as sums, $(2+1)+1=4$ and $2+1=3$, respectively.
These observations suggest the following procedure to predict multiplicities:  
\vskip .1in
\noindent
{\sc (Conjectured) Procedure to obtain exponents in singular-locus equation} 

\noindent
{\bf Input:} A linear compartmental model $\mathcal{M}=(G, In, Out, Leak)$ with 
input, output, and leak in compartment 1 only
($In=Out=Leak=\{1\}$), and such that $G$ is a bidirectional tree.

\noindent
{\bf Output:} One integer associated to each edge $(i,j)$ of $G$ (which is the purported multiplicity of $a_{ji}$ in the singular-locus equation of $\mathcal{M}$).
 
\noindent
{\bf Steps:} 
\begin{itemize}
	\item {\em Part 1: outgoing edges} (directed away from compartment 1)
		\begin{enumerate}
		\item Label each outgoing leaf-edge with 0.
		\item As long as there are unlabeled outgoing edges, consider an outgoing edge $(i,j)$ such that all outgoing edges of the form $j \to \star$ have already been labeled.  Add 1 to each of these labels, and then compute their sum $S$.  
		Label edge $(i,j)$ with $S$.
		\end{enumerate}
	\item {\em Part 2: incoming edges} (directed toward compartment 1)
		\begin{enumerate}
		\item Label each incoming leaf-edge with 1.
		\item As long as there are unlabeled incoming edges, consider an incoming edge 
		$(j,i)$ such that all incoming edges of the form $\star \to j$ have already been labeled.  Label the edge $(j,i)$ with 1 plus the sum of the labels of all edges incoming to $j$.
		\end{enumerate}
\end{itemize}

The above procedure and the following conjecture are due to Molly Hoch, Mark Sweeney, and Hwai-Ray Tung (personal communication).

\begin{conjecture}[Tree conjecture] \label{conj:tree}
The procedure above yields the multiplicities of parameter variables $a_{ji}$ in the equation of the singular locus.
\end{conjecture}
\noindent
Hoch, Sweeney, and Tung verified that the conjecture holds for trees on up to 4 nodes.

\section{Identifiability degree: mammillary, catenary, and cycle models} \label{sec:iddegree}

In this section, we discuss the identifiability degrees of mammillary, catenary, and cycle models (Proposition~\ref{prop:prior-ident-degree} and Theorem~\ref{thm:ident-degree-cycle}). 
The identifiability degree, a term we introduce here, is $m$ if exactly $m$ sets of parameter values map to a generic input-output data vector:

\begin{definition} \label{def:id-deg}
The {\em identifiability degree} of a (generically locally identifiable) model is $m$ if the coefficient map is generically $m$-to-1.
\end{definition}
\noindent
In other words, the identifiability degree is the number of elements in the fiber of the coefficient map over a generic point.

Cobelli, Lepschy, and Romanin Jacur~\cite{cobelli-constrained-1979} showed that 
the identifiability degrees of mammillary and catenary models are, respectively, $(n-1)!$ and 1:
\begin{proposition}[Mammillary and catenary~\cite{cobelli-constrained-1979}] \label{prop:prior-ident-degree}
Assume $n \geq 2$.
\begin{enumerate}

\item The identifiability degree of the 
mammillary (star) model in Figure~\ref{fig:cycle} is 
 $(n-1)!$, where $n$ is the number of compartments.

\item   The identifiability degree of the 
catenary (path) model in Figure~\ref{fig:cat} is 
$1$.  That is, the model is generically globally identifiable.

\end{enumerate}
\end{proposition}
\noindent
Cobelli, Lepschy, and Romanin Jacur also computed the identifiability degrees for versions of the
  mammillary and catenary models in which the input/output compartment need not be, respectively, the central compartment or an ``end'' compartment of the path~\cite{cobelli-constrained-1979}.  
 
Here we prove that the identifiability degree of a cycle model is $(n-1)!$.  
\begin{theorem}[Cycle] \label{thm:ident-degree-cycle}
Assume $n \geq 3$.
The identifiability degree of the 
cycle model in Figure~\ref{fig:cycle} is 
 $(n-1)!$, 
 where $n$ is the number of compartments.  
\end{theorem}

\begin{proof} 
Let $E_{j}:=E (a_{32},a_{43}, \dots,a_{1,n} )$ be the $j$-th elementary symmetric polynomial on the parameters
$a_{32}, a_{43}, \dots,a_{1,n}$.

Recall from the proof of Theorem~\ref{thm:cycle}, specifically, equations~\eqref{eq:coef-cycle-1}--\eqref{eq:coef-cycle-2}, that 
the coefficients on the left-hand side of the input-output equation are 
	$c_0 = a_{01} E_{n-1}$   and 
	$c_i = \left(a_{01}+a_{21}\right) E_{n-i-1}  +  E_{n-i-2}$
			  (for  $i = 1,2,\dots, n-1$), and those on the right-hand side are
	$d_i =  E_{n-i-1}$ (for $i=0,1,\dots, n-2$).
These coefficients  $c_i$ and $d_i$ are invariant under permutations of the $a_{32}, a_{43}, \dots,a_{1,n}$, so the identifiability degree is at least $(n-1)!$.	
	
We also see that the coefficients 
are related by the equation $c_0 = a_{01}d_0$ and as follows:
\begin{align*}
c_1 &= (a_{01}+a_{21})d_1 + d_0 \\
c_2 &= (a_{01}+a_{21})d_2 + d_1 \\
& ~~ \vdots \\
c_{n-1}&=(a_{01}+a_{21})d_{n-1}+d_{n-2}~.
\end{align*}
%
Thus, $a_{01}=c_0/d_0$ and 
$a_{21}=(c_{n-1}-d_{n-2})/d_{n-1} - c_0/d_0$, so both $a_{01}$ and $a_{21}$ can be uniquely recovered (when the parameters $a_{ij}$ are generic).  

Now consider the remaining coefficients $a_{32}, a_{43}, \dots,a_{1,n}$.  We may assume, by genericity, that these $a_{ij}$'s are distinct.
Having proven that the 
identifiability degree is {\em at least} $(n-1)!$, we need only show that the set 
$\mathcal{A}:=\{ a_{32}, a_{43}, \dots,a_{1,n} \}$ (of size $n-1$) can be recovered from the coefficients $c_i$ and $d_i$ 
(because this would imply that the identifiability degree is {\em at most} $(n-1)!$).
To see this, first recall that the $d_i$'s comprise all the elementary symmetric polynomials, from $E_1$ to $E_{n}$, on 
$a_{32}, a_{43}, \dots,a_{1,n}$, and these $E_i$'s are, up to sign, the coefficients of the following (monic) univariate polynomial:
	\[
	\prod_{i=2}^n (x-a_{i+1,i})~.
	\]
In turn, a monic polynomial in $\mathbb{R}[x]$
is uniquely determined by its set of roots (in $\mathbb{C}$),
so the set $\mathcal{A}$ is uniquely determined by the $d_i$'s.  This completes the proof.
%
\end{proof}

The proof of Theorem~\ref{thm:ident-degree-cycle} showed that for the cycle model, under generic conditions, the parameters $a_{01}$ and $a_{21}$ can be uniquely recovered from input-output data, but only the {\em set} of the remaining parameters $\{ a_{32}, a_{43}, \dots,a_{1,n} \}$ can be identified.  
Also, this set does {\em not} reflect some underlying symmetry in the cycle model in Figure~\ref{fig:cycle} (the symmetry of the cycle is broken when one compartment is chosen for the input/output/leak).  
In contrast, the identifiability degree of the mammillary model, which is $(n-1)!$ (Proposition~\ref{prop:prior-ident-degree}), {\em does} reflect 
the symmetry of its graph: the $n-1$ non-central compartments can be permuted (see Figure~\ref{fig:cycle}).

\section{Discussion} \label{sec:discussion}
In this work, we investigated, for linear compartmental models, the set of parameter values that are non-identifiable.  Specifically, we focused on examples where a single determinantal equation, the singular-locus equation, defines the set.  We showed first that this equation gives information about which submodels are identifiable, and then computed this equation for cycle and mammillary (star) models.  These equations revealed that when the parameters are known to be positive, then these parameters can be recovered from input-output data, as long as certain pairs of parameters are not equal.  We also stated an in-depth conjecture (Conjecture \ref{conj:tree}) regarding the singular-locus equation for tree models.  While we focused on three specific families, it would be interesting to explore additional families, especially families where the singular locus is defined by a single equation, which brings us back to a question discussed in Section \ref{sec:background}:

\begin{question}
When is the locus of non-identifiable parameter values codimension one?
\end{question}

Another topic we examined is the identifiability degree, the number of parameter sets that map to a generic input-output data vector.  We computed this degree for cycle models, and noted that the degree was already proven for catenary (path) and mammillary models~\cite{cobelli-constrained-1979}.

A natural future problem is to investigate how our results change when the input, output, and/or leak are moved to other compartments.  As mentioned earlier, results in this direction  for the identifiability degree were obtained by Cobelli, Lepschy, and Romanin Jacur for catenary and mammillary models~\cite{cobelli-constrained-1979}. 
We also are interested in the effect of adding more inputs, outputs, or leaks. 
Some results in this direction, which pertain to identifiability, are given in~\cite{Gerberding-Obatake-Shiu,linear-i-o}.

Finally, in the course of our investigation of cycle models, we were able to add to the list of models that we can conclude  are identifiable simply from inspecting the underlying graph.  Restricting our attention to models in which a single compartment has an input, output, and leak (and no other compartment has an input, output, or leak) and adding our results from this paper, the only such models that are known to be identifiable by inspecting their underlying graph are:
\begin{enumerate}
	\item models in which the underlying graph is inductively strongly connected~\cite[Theorem~1]{MeshkatSullivantEisenberg}, 
	such as tree models 
	(e.g., catenary and mammillary models)~\cite{cobelli-constrained-1979}, and 
	\item cycle models (Theorem~\ref{thm:cycle}). 
\end{enumerate}
A final direction worth exploring is whether one can add to this list using the combinatorial interpretation of the input-output coefficients we gave in Section \ref{sec:jac} (Theorem~\ref{thm:coeff-i-o-general}).  
In fact, this avenue has already proven fruitful, as a recent extension of Theorem~\ref{thm:coeff-i-o-general} allowed Bortner {\em et al.} to completely characterize identifiable tree models (e.g., catenary and mammillary models) with one input, one output, and any number of leaks~\cite{BGMSS}.

\subsection*{Acknowledgements}
This project began at a SQuaRE (Structured Quartet Research Ensemble) at AIM, and the authors thank AIM for providing financial support and an excellent working environment. The authors thank 
Luis Garc\'{\i}a Puente and 
Heather Harrington for their insights and feedback throughout the course of this project, and acknowledge a referee for helpful comments. EG was supported by the NSF (DMS-1620109). NM was partially supported by the Clare Boothe Luce Program from the Henry Luce Foundation and by the NSF (DMS-1853525).  AS was supported by the NSF (DMS-1312473/1513364). 
Molly Hoch, Mark Sweeney, and Hwai-Ray Tung posed  
Conjecture~\ref{conj:tree} as a result of research conducted in the NSF-funded REU in the Department of Mathematics at Texas A\&M University (DMS-1460766), in which AS served as mentor.  

\appendix

\section{Proof of Proposition~\ref{prop:coeff}} \label{sec:appendix}
Here we prove Proposition~\ref{prop:coeff}, which, for convenience, we restate here:
\begin{proposition}[Proposition \ref{prop:coeff}]
For a linear compartmental model with $n$ compartments, let $A$ be the compartmental matrix.
Write the characteristic polynomial of $A$ as:
	\begin{align*}
	\det (\lambda I - A) ~=~ \lambda^n + e_{n-1} \lambda^{n-1} + \cdots + e_0~.
	\end{align*}
Then $e_i$ (for $i=0,1,\dots,n-1$) is the sum -- over $(n-i)$-edge, spanning, incoming forests of
the model's leak-augmented graph  
 $\widetilde G$ -- of the product of the edge labels in the forest:
	\begin{align} \label{eq:coeff-forest}
			e_i ~&=~
					\sum_{F \in \mathcal{F}_{n-i}( \widetilde G)} \pi_F   ~.
	\end{align}
\end{proposition}

To prove Proposition~\ref{prop:coeff}, we need a closely related result, Proposition~\ref{prop:buslov} below, which is 
due to Buslov~\cite[Theorem 2]{buslov}.  For that result, recall that the {\em Laplacian matrix} of a graph $G$ with $n$ vertices and edges $i \to j$ labeled by $b_{ji}$ is the $(n \times n)$-matrix $L$ with entries as follows: 
\begin{align*}
L_{ij} ~:=~
 \begin{cases} 
	-b_{ji} 					& \text{if } i \neq j\\
	\sum_{k \neq i} b_{ki} 		& \text{if } i = j~. \\
   \end{cases}
\end{align*}
For a model $\mathcal{M}=(G,In,Out,Leak)$ with no leaks ($Leak=\emptyset$), the Laplacian matrix of $G$ is the transpose of the negative of the
compartmental matrix of $\mathcal{M}$.
\begin{proposition}[Buslov] \label{prop:buslov}
Let $L$ be the Laplacian matrix of a directed graph $G$ with $n$ edges. 
Write the characteristic polynomial of $L$ as:
	\begin{align*}
	\det (\lambda I - L) ~=~ \lambda^n + e_{n-1} \lambda^{n-1} + \cdots + e_0~.
	\end{align*}
Then
	\begin{align*}
			e_i ~&=~ (-1)^{n-i}  \sum_{F \in \mathcal{F}_{n-i}( G)} \pi_F   \quad \quad  \text{ for $i=0,1,\dots,n-1$. }
	\end{align*}
In particular, $e_0=0$.
\end{proposition}

\begin{remark} \label{rmk:buslov}
Buslov's statement of Proposition~\ref{prop:buslov} (namely, \cite[Theorem 2]{buslov}) differs slightly from ours: it refers to forests with $i$ connected components rather than $n-i$ edges.  That version is equivalent to ours, because a spanning, incoming forest of a graph $G$ with $n$ vertices has $i$ connected components if and only if it has $n-i$ edges.
\end{remark}

\begin{remark} \label{rmk:connection-to-laplacian-results}
Propositions~\ref{prop:coeff} and~\ref{prop:buslov}
are closely related to the all-minors matrix-tree theorem~\cite{chaiken}.
The all-minors matrix-tree theorem is a formula for the minors of the Laplacian matrix of a (weighted, directed) graph $G$, and it is a sum over certain forests in $G$.
\end{remark}

We obtain the following consequence of Proposition~\ref{prop:buslov}:
\begin{proposition} \label{prop:buslov-neg-lap}
Let $L$ be the Laplacian matrix, or its transpose, of a directed graph $G$ with $n$ edges. 
Write the characteristic polynomial of $-L$ as:
	\begin{align*}
	\det (\lambda I + L) ~=~ \lambda^n + e_{n-1} \lambda^{n-1} + \cdots + e_0~.
	\end{align*}
Then 
	\begin{align*}
			e_i ~&=~ 
				 \sum_{F \in \mathcal{F}_{n-i}( G)} \pi_F     \quad \quad  \text{ for $i=0,1,\dots,n-1$. }
	\end{align*}
\end{proposition}

\begin{proof}
Let $p(\lambda):=\det (\lambda I - L)$.
The result follows directly from 
	Proposition~\ref{prop:buslov},
the equality $\det (\lambda I + L) = (-1)^n \cdot p(-\lambda)$, and
(for the case of the transpose) the invariance of the determinant under taking transposes.
\end{proof}

We can now prove Proposition~\ref{prop:coeff}.
\begin{proof}[Proof of Proposition~\ref{prop:coeff}]
Let $A$ be the compartmental matrix of a model $\mathcal{M}$ with $n$ compartments and $k$ leaks (so, $1 \leq k \leq n$). 
	The case of no leaks ($k=0$ and any $n$) is known, by Proposition~\ref{prop:buslov-neg-lap}: $A$ is the transpose of the negative of the Laplacian of the graph $G$ of $\mathcal{M}$.  Also, the case of $k=1$ leak and $n=1$ compartment is straightforward: the compartmental matrix is $A=[a_{01}]$, so $\det (\lambda I -A)= \lambda + a_{01}$, which verifies equation~\eqref{eq:coeff-forest}.

The above cases form the base cases for proving equation~\eqref{eq:coeff-forest} by strong induction on $(n,k)$.  Next we prove the inductive step (i.e., the $(n,k+1)$ case, assuming all smaller cases):

\smallskip
\noindent
{\bf Claim}: Equation~\eqref{eq:coeff-forest} holds for all linear compartmental models with $n$ compartments and $k+1$ leaks (where $n \geq 2$ and $1 \leq k \leq n-1$), assuming that the equation holds for all linear compartmental models with $(\widetilde n, \widetilde k)$ for which either $\widetilde n=n$ and $\widetilde k  \leq k$, or $\widetilde n \leq n-1$ and $\widetilde k \leq \widetilde n$.

Relabel the compartments so that compartment 1 has a leak (labeled by $a_{01}$), and, also, there is a (directed) edge from compartment 1 to compartments $2, 3, \dots, m$ and {\em not} to compartments $m+1, m+2, \dots, n$ (for some $1 \leq m \leq n$).  This  operation permutes the rows and columns of $A$ in the same way, which does not affect the characteristic polynomial $\det (\lambda I - A)$.

We compute as follows:
\begin{align}
\notag
\det (\lambda I - A) &~=~
\det
\left[ 
\begin{array}{c@{}c}
 \left[\begin{array}{c}
         \lambda + (a_{21}+ \cdots + a_{m1})+a_{01} \\      \end{array}\right] & \star  \\     
         \star & \left[\begin{array}{c}
                       \lambda I - M   \\                          \end{array}\right]  \\   
\end{array}\right] \\
\notag 
&~=~ 
\det 
\left[ 
\begin{array}{c@{}c}
 \left[\begin{array}{c}
         \lambda + (a_{21}+ \cdots + a_{m1}) \\      \end{array}\right] & \star  \\     
         \star & \left[\begin{array}{c}
                       \lambda I - M   \\                          \end{array}\right]  \\   
\end{array}\right]
+
\det 
\left[ 
\begin{array}{c@{}c}
 \left[\begin{array}{c}
         a_{01} \\      \end{array}\right] & \star  \\     
         \star & \left[\begin{array}{c}
                       \lambda I - M   \\                          \end{array}\right]  \\   
\end{array}\right]
\\
\label{eq:2-dets}
&~=~ 
\det (\lambda I -N)
+
a_{01} \det (\lambda I -M)~,
\end{align}
where $N$ is the compartmental matrix for the model obtained from $\mathcal{M}$ by removing the leak from compartment 1 (so, this model has $k$ leaks and $n$ compartments), and (by Lemma~\ref{lem:remove-1})
$M$ is the compartmental matrix for an $(n-1)$-compartment model $\mathcal{M}_1$ for which the leak-augmented graph is $\widetilde G_1$.

By the inductive hypothesis, the coefficient of $\lambda ^i$ in the first summand in \eqref{eq:2-dets} is:
	\begin{align} \label{eq:summand-1}
	 \sum_{F \text{ an } {(n-i)}\text{-edge, spanning, incoming forest of } \widetilde G \text{ {\em not} involving } a_{01}} \pi_{F}~. 
	\end{align}

Also by our inductive hypothesis, the coefficient of $\lambda ^i$ in the second summand in \eqref{eq:2-dets} is a sum over $(n-i-1)$-edge forests in $\widetilde G_1$. Specifically, that coefficient is:
	\begin{align} \label{eq:summand-2}
	a_{01}  \sum_{F \in \mathcal{F}_{n-i-1}( \widetilde G_1)} \pi_{F'} ~=~
	 \sum_{F \text{ an } {(n-i)}\text{-edge, spanning, incoming forest of } \widetilde G \text{ that involves } a_{01}} \pi_{F} ~.
	\end{align}
Taken together, equations~\eqref{eq:summand-1} and~\eqref{eq:summand-2} prove the Claim.
\end{proof}

\bibliographystyle{plain}
\bibliography{square}
\end{document}